\documentclass[reqno,11pt]{amsart}
\usepackage{amsmath,amssymb,amsfonts,amsbsy,setspace, amsthm,}

\usepackage[utf8]{inputenc}

\makeatletter

\everymath{\displaystyle} 

\theoremstyle{plain}
\newtheorem{thm}{Theorem}[section]
\newtheorem{lmm}[thm]{Lemma}
\newtheorem{crl}[thm]{Corollary}
\numberwithin{equation}{section} 
\numberwithin{figure}{section} 
\newtheorem{prp}[thm]{Proposition}
\theoremstyle{remark}
\newtheorem{rmk}[thm]{Remark}
\newtheorem{xpl}[thm]{Example}
\newtheorem*{acknowledgement*}{Acknowledgement}
\theoremstyle{definition}
\newtheorem{dfn}[thm]{Definition}

\def\eps{\varepsilon}
\newcommand{\R}{\mathbb R}

\newcommand{\bV}{\overline V}
\newcommand{\bO}{\overline \Omega}
\newcommand{\pV}{\partial V}
\newcommand{\pO}{\partial \Omega}

\newcommand{\Winf} {W^{1,\infty}}

\newcommand{\phizero}{\varphi^{0}}

\def\esup{\mathop\mathrm{ess\,sup\,}}
\newcommand{\Leb}{\mathcal{L}}

\newcommand{\Rd}{\mathbb{R}^{d}}

\newcommand{\pathbar}{\overline{\mathsf{path}}}
\newcommand{\Sp}{S^{+}}
\newcommand{\Sm}{S^{-}}

\newcommand{\K}{K(x)}

\newcommand{\grad}{\nabla}
\newcommand{\gradu}{\nabla u}
\newcommand{\WinfCo}{W^{1,\infty}(\Omega)\cap C(\overline{\Omega})}
\newcommand{\interior}{\mathrm{int} \,}
\def\argmax{\mathop\mathrm{arg\,max\,}}

\usepackage{color}
\usepackage{hyperref}
\usepackage{enumerate}
\usepackage{eucal}


\begin{document}
	\title[]{On functions with given boundary data and convex constraints on the gradient}
	\author{Camilla Brizzi}
	\address{{\bf C.B.} Dipartimento di Matematica ed Informatica,
		Universit\'a di Firenze, Viale Morgagni, 67/a - 50134 Firenze, ITALY}
	\begin{abstract}
		Let $\Omega\subset\Rd$ an open set. Given a boundary datum $g$ on $\pO$ and a function $K:\bO\to\mathcal{K}$, the family of all compact convex sets of $\Rd$, we prove the existence of functions $u:\Omega\to\R$ such that $u=g$ on $\pO$ and $\gradu(x)\in K(x)$ a.e. and we investigate the regularity of such solutions on the set $\mathcal{U} \subset \bO$ of points at which they all coincide. 
	\end{abstract}

	\date{\today}
	\maketitle	
	\section{Introduction}
	In this article we consider the problem of interpolating a given real boundary values $g$ into a connected and bounded open set $\Omega\subset\Rd$ under a pointwise constraint on the gradient, i.e. let $g\in C(\pO)$, we consider the solutions $u\in \Winf(\Omega)\cap C(\bO)$ of the following system
	\begin{equation}
		\label{problem}
		\begin{cases}
			u=g \quad &\mbox{on} \ \pO,\\
			\nabla u(x)\in K(x) \quad &\mbox{for a.e. }x \in \Omega,
		\end{cases}
		\tag{P}
	\end{equation}
	where $K(x)$ belongs to $\mathcal{K}$, the family of compact convex sets included in $\Rd$ and $K(x)$ is continuous w.r.t. the Hausdorff distance. This is a classical problem and the case $K(x):=B(0,f(x))$ has been studied by Aronsson in \cite{Aro2009}.\\
	In order to study existence and some properties of solutions of problem \eqref{problem}, we endow $\Omega$ with a metric that depends on $K(x)$. More precisely, 
	for each $x\in\bO$, we define the \textit{Minkowski functional} $\varphi(x,\cdot)$ associated to the closed convex set $K(x)$ and its \textit{support function} $\phizero(x,\cdot)$:
	\begin{equation*}
		\phizero(x,q):=\sup\left\{p\cdot q \ | \ p\in K(x) \right\},
	\end{equation*} 
	for which we prove continuity and the existence of an optimal $p$ for every $x\in\bar{\Omega}$ and $q\in\Rd$ (see Proposition \ref{phizeropuntuale} and Proposition \ref{continuityphizero}). We also give a further result presented in Proposition \ref{differentiabilityphizero} showing that under the assumption of some regularity of $\varphi$ and of strict convexity of $K(x)$ for every $x$ one can obtain $\phizero(\cdot,q)\in C^{1}(\Omega)$.\\ 
	The reason to consider such a functional is that $\phizero$ is a Finsler metric (for a more general definition of Finlser metric and more details, see \cite{DecPal1995,BaoCheShe2000}), so that we can consider the associated  non symmetric distance $d=d_{\phizero}$, defined for every $x,y\in\Omega$, as follows
	\begin{equation}
		\label{d}
		d(x,y):=\inf\left\{\int_{0}^{l(\gamma)}\phizero(\gamma(s),\dot{\gamma}(s))ds: \gamma\in\mathsf{path}(x,y)\right\},
	\end{equation}
	where $\mathsf{path}(x,y)$ is the set
	\begin{equation*}
		\{\gamma\in \Winf((0,l(\gamma)), \Omega)\cap C([0,l(\gamma)],\Omega) \ : \ \gamma(0)=x, \ \gamma(l(\gamma))=y, |\dot{\gamma}(s)|=1 \ \mbox{a.e.}\}.
	\end{equation*}
	As showed in Definition \ref{distance}, this pseudo-distance can be extended also to $\pO$. Moreover we provide a proof (see Remark \ref{equivalenced}) of the already known fact: 
	\begin{equation}
		\label{finslerdistance}
		d(x,y):=\sup\left\{u(y)-u(x): u\in\Winf(\Omega)\cap C(\bO), \ \gradu(x)\in\K \ \mbox{a.e. in }\Omega \right\},
	\end{equation}
	where the left hand side of the equality is a distance introduced in \cite{DecPal1995}.
	\\ The existence of a solution for \eqref{problem} is not trivial in general. However, if we assume that the boundary datum $g$ is $1$-lipschitz w.r.t. the distance $d$, the problem \eqref{problem} always admits a solution. Indeed in Proposition \ref{S+S-lipschitz} and Corollary \ref{S+S-gradient} we prove that the functions 
	$S^{-}$ and $S^{+}$, respectively called \textit{maximal} and \textit{minimal extension}, defined for any $x\in\bO$ by:
	\begin{align*}
		&S^{-}(x)=\sup\left\{g(y)-d(x,y): y\in\pO\right\},\\
		&S^{+}(x)=\inf\left\{g(y)+d(y,x): y\in\pO\right\},
	\end{align*}
	satisfy  $S^{+}(y)=S^{-}(y)=g(y)$ for all $y\in\pO$ and $\nabla S^{+}(x), \nabla S^{-}(x)\in K(x)$ for a.e. $x$. For every $u$ solution of \eqref{problem} (Proposition \ref{solutions}), 
	\begin{equation*}
		S^{-}\le u \le S^{+}.
	\end{equation*}
	The proofs of these facts are inspired by the proofs of Proposition 2.7, 2.8 and 2.9 in \cite{ChaDep2007}. \\
	Moreover if there exists a solution $u\in \Winf(\Omega)\cap C(\bO)$ of \eqref{problem}, then $g$ turns out to be lipschitz w.r.t. the distance $d$ (see Proposition \ref{1lipu} and Remark \ref{lipurmk}).\\
	In this paper we are interested in the points of $\Omega$ on which the upper and lower solutions coincide and we present some results that extend to more general $K(x)$ the results proven by Aronsson in \cite{Aro2009} for the case $K(x):=B(0,f(x))$, with $f\in C^{1}(\Omega)$ and $f$ bounded from below and above.
	To this aim, in Proposition \ref{existencecurve} and Proposition \ref{existencecurvebondary}, we first show the existence of optimal curves for \eqref{d}, namely we prove that there exists $\gamma$ such that $d(x,y)=\int_{0}^{l(\gamma)}\phizero(\gamma(s),\dot{\gamma}(s))ds$.\footnote{The proof we present makes use of simple tools from functional analysis. However, the existence of the minimizing curve can be inferred in a more standard way by Teorem 4.3.2 of \cite{AmbTil2004}.}
	A consequence of this result (see Proposition \ref{maxminextensioncurves}) is that also the definitions of $S^{+}$ and $S^{-}$ admit respectively a minimum and a maximum, i.e. there exist $y_{1},y_{2}\in\pO$ and $\gamma_{1},\gamma_{2}$, such that:
	$S^{+}(x)=g(y_{1})+\int_{0}^{l(\gamma_{1})}\phizero(\gamma_{1}(s),\dot{\gamma}_{1}(s))ds$  and
	$S^{-}(x)=g(y_{2})-\int_{0}^{l(\gamma_{2})}\phizero(\gamma_{2}(s),\dot{\gamma}_{2}(s))ds$.\\
	The main result is presented in Theorem \ref{curveofcoincidence} where we show that if there exists a point $x_{0}\in\Omega$ such that $\Sp(x_{0})=\Sm(x_{0})$, then there exists a curve, all contained in $\Omega$ except for the extreme points, that connects two points of the boundary, along which $\Sp$, $\Sm$ and every solution $u$ of \eqref{problem} coincide. Such a curve can be reparameterized  to be $1$-lipschitz and it is optimal for \eqref{d}. Moreover, all the solutions of \eqref{problem} are derivable $\mathcal{H}^{1}$-a.e. along that curve with curve derivative in each point equal to $\phizero$ in that point. We refer to the set of the points where $\Sp=\Sm$ as \textit{uniqueness set}.
	Under suitable assumption on $\phizero$ or $d$ it is possible to prove (see Proposition \ref{semiconcavity1} and Proposition \ref{semiconcavity2}) that the maximal and minimal extensions are respectively locally semiconvex and semiconcave. Thanks to this fact we are able to prove a further regularity result (Th. \ref{regularity}): every solution $u$ of \eqref{problem} is continuously differentiable on the uniqueness set. \\
	
	In the literature there are several examples of variational functionals with constraints on the gradient. An interesting example is the model of the dielectric breakdown (see \cite{ChaDep2001,GaNePo2001}), in which a body, subject to an electric field $\nabla u$, behaves as an insulator if $\nabla u(x)$ belongs to a convex set $K(x)$ for a.e. $x$, otherwise the dielectric breakdown occurs and the body starts to conduct. 
	Another important application is that the problem \eqref{problem} can be restated in the framework of supremal functionals of the form
	\begin{equation}
		\label{supfunctional}
		F(u):=\esup_{x\in\Omega}H(x,\nabla u(x)), \quad u\in\WinfCo.
	\end{equation}
	Indeed, as it has been shown first in \cite{BarJenWan2001}, natural assumptions to have weak* lower semicontinuity of $F$ are the lower semicontinuity and quasi-covexity (convexity of the sublevel sets) w.r.t. the second variable of the supremand $H$. Then, requiring also the coercivity of the functional $H(x,\cdot):\Rd\to\R$, as shown in Section \ref{section 5},if one calls
	\begin{equation*}
		\label{varprob}
		\mu:=\min\left\{F(v,\Omega):=\esup_{x\in\Omega}H(x,\nabla v(x)) \ : \ v\in g+W^{1,\infty}(\Omega)\cap C_{0}(\Omega)\right\},
	\end{equation*}
	finding a minimizer for the supremal variational problem above is equivalent to solving
	\begin{equation*}
		\begin{cases}
			u=g \quad &\mbox{on} \ \pO,\\
			H(x,\nabla u(x))\le\mu \quad &\mbox{for a.e. }x \in \Omega,
		\end{cases}
	\end{equation*}
	that is of the form of \eqref{problem}, thanks to the assumptions on $H$. In  the case considered by Aronsson in \cite{Aro2009}, for example, the supremal functional associated is 
	\begin{equation*}
		F(u):=\esup_{x\in\Omega}\frac{|\nabla u(x)|}{f(x)}, \quad u\in\WinfCo.
	\end{equation*} 
	One first result of this application (see Proposition \ref{uniqincludedattain}) is that if $x$ is a point of the uniqueness set and $u$ is a solution of \eqref{supfunctional} that is therein differentiable, then $H(x,\nabla u(x))=\esup_{x\in\Omega}H(x,\gradu(x))$. Moreover, as showed in \cite{BriDep2022}, it is possible to extend the definition of $H(x,\nabla u(x))$ to the points where $u$ is not differentiable (see Definition \ref{punctualextension}) and to prove the same equality for every $x$ in the uniqueness set (see Proposition \ref{uniqincludedattainbis}). \\
	Supremal functionals of the form \eqref{supfunctional} have been widely studied in the last years (see for instance \cite{Aro1965, Aro1966, GarWanYu2006,CraWanYu2009,Yu2006,ChaDep2007,Pri2008} and many others) and they can be seen as a generalization of the more known case $H(x,p)=|p|$, whose notoriety comes by the fact that the \textit{absolute minimizers} of the supremal variational problem associated are solutions of the well known $\infty$-\textit{Laplacian} equation
	\begin{equation*}
		\Delta_{\infty} u:=\left< Du,D^{2}uDu\right>=\sum_{i,j=1}^{n}u_{x_{i}}u_{x_{j}}u_{x_{i}x_{j}}=0.
	\end{equation*}
	Concerning this topic we refer to \cite{Aro1967, CraEvaGar2001, AroCraJuu2004} and the references therein. Absolute minimizers for a problem like \eqref{supfunctional} are, roughly speaking, functions that minimize the functional $F$ in any open set well contained in $\Omega$ (see Definition \ref{AM} for the rigorous definition) and due to the non-convex and non-local nature of supremal variational problems, it turns out to be an important a widely studied class of minimizers. In our setting, under the same assumptions of the ones used in \cite{BriDep2022}, it is possible to improve Proposition \ref{uniqincludedattainbis} proving that if $u$ is an absolute minimizer, also the other inclusion holds, that means that the uniqueness set exactly coincides with the set of points $x$ where $H(x,\nabla u(x))=\esup_{x\in\Omega}H(x,\gradu(x))$ (see Theorem \ref{attainincludeduniq}). This last set has its own interest in the framework of supremal functionals because it provides a further minimality property of the absolute minimizers (see \cite{BriDep2022}). \\
	Finally, convex constraints on the gradient appear also in the Monge- Kantorovich transportation problem, in which one seeks to maximize
	\begin{equation*}
		\int u(x)d\mu(x)- \int u(x)d\nu(x),
	\end{equation*} 
	among all the $u$'s such that $\gradu(x)\in K(x)$ a.e. (see for instance \cite{Vil2009}).
	\\
	This paper is organized as follows: Section \ref{section2} is devoted to some preliminary definitions and results about regularity properties of the \textit{support function} $\phizero$ and the distance $d$. Furthermore we show the existence of optimal curves for $d$. In Section \ref{section3} we present the \textit{maximal} and \textit{minimal extension} functions, proving some properties of the solutions of \eqref{problem} and giving a representation result for $S^{+}$ and $S^{-}$. The last part of the section is dedicated to Theorem \ref{curveofcoincidence} which provides a structure result concerning the set where $S^{+}$ and $S^{-}$ coincide and a small regularity result about $S^{+}$ and $S^{-}$ on this set. In Section \ref{section4} we recall the definition of semiconcavity and semiconvexity and in Proposition \ref{semiconcavity1} and \ref{semiconcavity2} we prove that under suitable assumptions $S^{+}$ and $S^{-}$  are respectively locally semiconcave and locally semiconvex. This fact allows for Theorem \ref{regularity}, where it is showed that every solution $u$ of \eqref{problem} is continuously differentiable on the set where $S^+$ and $S^-$ coincide. In Section \ref{section 5} we apply the results of the previous sections to the case of supremal functionals defined by \eqref{supfunctional}. Finally, in the Appendix \ref{appendix}, we provide the proof of some results which are useful in Section \ref{section2}.
	
	\section{Existence of minimizing curves} \label{section2}	
	Let $\Omega$ be a bounded  and connected open set of $\mathbb{R}^{d}$. For any $x\in\bO$, $K(x)$ will be a compact convex set. 
	\begin{dfn}
		For any  $x\in\bO$, we define the \textit{Minkowski functional} associated to  $K(x)$, as:
		\begin{equation*}
			\varphi(x,p):=||p||_{K(x)}=\inf\left\{t>0 \, : \, \frac{p}{t}\in K(x)\right\}.
		\end{equation*}
		
		It is not difficult to prove that:
		\begin{enumerate}[(i)]\item $\varphi(x,p)\ge 0$ for every $(x,p)\in(\bO\times\Rd)$ and $\varphi(x,p)=0$ if and only if $p=0$;
			\item $\varphi(x,r p)=r\varphi(x,p)$ for all $r>0$;
			\item $\varphi(x,p_{1}+p_{2})\le\varphi(x,p_{1})+\varphi(x,p_{2})$, for all $p_{1},p_{2}\in\Rd$.
		\end{enumerate} 
		That is why $\varphi$ is not exactly a norm: the property (ii) tells that $\varphi$ is positively $1$-homogeneous but not absolutely. Indeed it can happen that $\varphi(x,p)\not=\varphi(x,-p)$. However thanks to (ii) and (iii) we still have the convexity of $\varphi(x,\cdot)$.
	\end{dfn}
	\begin{rmk}
		It holds:
		\begin{align*}
			&p\in K(x) \Leftrightarrow \varphi(x,p)\le 1\\
			&p\in\partial K(x)\Leftrightarrow \varphi(x,p)=1.
		\end{align*}
	\end{rmk}
	We will also requires that:
	\begin{enumerate}[(a)]
		\item \label{alphaM} there exist $\alpha,M>0$ such that $B(0,\alpha)\subset K(x)\subset B(0,M)$, for every $x\in\bO$;
		\item \label{continuityphi} $\varphi(\cdot,p)$ is continuous for every $p\in\Rd$.
	\end{enumerate}
	\begin{rmk}
		Assuming the continuity of $\varphi(\cdot,p)$ in \eqref{continuityphi} is equivalent to asking for the continuity of $K:\Omega\to\mathcal{K}$ w.r.t. the Hausdorff distance, as shown in Proposition \ref{hausdorffdistance} below.
	\end{rmk} 
	First of all we prove that $\varphi$ is continuous also w.r.t. the second variable.
	
	\begin{prp}
		\label{continuityphip}
		$\varphi(x,\cdot)$ is continuous w.r.t the Euclidean distance, uniformly w.r.t. $x\in\bO$.
	\end{prp}
	\begin{proof}
		We claim that \begin{equation}
			\label{normequivalence}
			\frac{|p|}{M}\le \varphi(x,p)\le\frac{|p|}{\alpha}, \quad \text{for all }x\in\bO. 
		\end{equation}
		For every $t<\frac{|p|}{M}$, $\frac{p}{t}\notin K(x)$, because $\left|\frac{p}{t}\right|$ is greater than $M$, that is not possible by assumption \eqref{alphaM}. Then the left inequality is proved. The proof of right one, comes easily by the fact that again by the assumption \eqref{alphaM}
		\begin{equation*}
			\alpha\frac{p}{|p|}\in K(x),
		\end{equation*}
	\end{proof}
	\begin{crl}
		\label{fullcontinuityphi}
		The functional $\varphi:\bO\times\Rd\to\R$ is continuous.
	\end{crl}
	\begin{proof}
		The proof is a direct consequence of Proposition \ref{continuityphip} and assumption \eqref{continuityphi}.
	\end{proof}
	
	\begin{prp}
		\label{hausdorffdistance}
		Let $x\in\bO$ and $(x_{n})_{n\in\mathbb{N}}$ a sequence converging to $x$. Then, for all $p\in\Rd$ 
		\begin{equation*}
			\varphi(x_{n},p)\longrightarrow\varphi(x,p)\iff K(x_{n})\overset{H}{\longrightarrow}K(x),
		\end{equation*}
		where on the right side we mean the convergence w.r.t. to the Hausdorff distance for sets.
	\end{prp} 
	\begin{proof}
		In order to lighten the notation, we write $K:= K(x)$ and $K_{n}:=K(x_{n})$.\\
		We first assume that $\varphi$ is continuous w.r.t $x$. 
		Let $\eps>0$, then there exists $n_{0}\in\mathbb{N}$ such that 
		\begin{equation}
			\label{epsilonphi}
			|\varphi(x_{n},p)-\varphi(x,p)|<\eps, \quad  \mbox{for all } n\ge n_{0}.
		\end{equation}
		We recall that $d_{H}(K_{n},K)=\max\left\{\rho(K_{n},K), \rho(K,K_{n}) \right\}$, where
		\begin{equation*}
			\rho(K_{n},K):=\sup_{p\in K_{n}} dist(p,K)\footnote{With this notation we mean the usual distance to a set w.r.t. the Euclidean distance, i.e. \[dist(p,K):=\inf_{q\in K}|p-q|\].} \ \mbox{and} \ \rho(K,K_{n}):=\sup_{p\in K} dist(p,K_{n}).\footnote{For more details about Hausdorff distance see e.g. \cite{HenMic2018}.}
		\end{equation*}
		
		Let $n\ge n_{0}$. Then for every $p\in K$ $\frac{p}{1+\eps}\in K_{n}$, since, by \eqref{epsilonphi}, $$\varphi(x_{n},p)<\varphi(x,p)+\eps\le1+\eps.$$ 
		That means that
		\begin{equation*}
			dist(p,K_{n})\le \left|p-\frac{p}{1+\eps}\right|=\frac{\eps|p|}{1+\eps}\le \frac{\eps}{1+\eps}M,
		\end{equation*}
		and thus $\rho(K,K_{n})<\frac{\eps}{1+\eps}M$. Let us now consider $p\in K_{n}$, then as above $\frac{p}{1+\eps}\in K$ and $\rho(K_{n},K)<\frac{\eps}{1+\eps}M$.\\
		We now prove the converse implication. Let $p\in\Rd$, proving that 
		\begin{equation*}
			\lim_{n\to+\infty}\varphi(x_{n},p)=\varphi(x,p)
		\end{equation*}
		is equivalent to show that 
		\begin{equation*}
			\lim_{n\to\infty}\varphi(x_{n},p)=\varphi(x,p)=1, \quad \mbox{for every }p\in\partial K.
		\end{equation*}
		Indeed for a generic $p\in\Rd$ such that $\varphi(x,p)=\bar{t}$, for some $\bar{t}>0$, the positive $1$-homogeneity of $\varphi$ implies 
		\begin{equation*}
			\varphi(x,p)=\bar{t}\Leftrightarrow\varphi(x,\frac{p}{\bar{t}})=1\Leftrightarrow\frac{p}{\bar{t}}\in\partial K.
		\end{equation*}
		Thus we restrict ourselves to $p\in\partial K$. Then, by convexity of $K$,
		\begin{equation*}
			\quad dist(\beta p,K)>0, \ \mbox{for every }\beta>1.
		\end{equation*} 
		Let $\{\beta_{k} \}$ a decreasing sequence of real numbers converging to $1$. Let $\eps>0$. Then for $k$ big enough,
		\begin{equation*}
			\eps>dist(\beta_{k}p,K)>\eta>0.
		\end{equation*} 
		Since $K_{n}{\overset{H}{\longrightarrow}}K$, for every fixed $k$, there exists $N_{k,\eta}\in\mathbb{N}$, such that 
		\begin{equation*}
			d(\beta_{k}p,K_{n})>\eta>0, \quad \mbox{for every } n\ge N_{k,\eps}\footnote{In order to see this one can use the fact that \[d_{H}(K_n,K)=\sup_{p\in\Rd}\left|dist(p,K_n)-dist(p,K)\right|\]},
		\end{equation*}
		which implies $\varphi(x_n,\beta_{k}p)>1$.
		
		By the triangular inequality, 
		\begin{equation*}
			\varphi(x_{n},p)\ge	\varphi(x_{n},\beta_{k}p)-\varphi(x_{n},\beta_{k}p-p)\ge 1 - \frac{|\beta_{k}p-p|}{\alpha},
		\end{equation*}
		where the last inequality is due to Proposition \ref{continuityphip}.
		We infer that 
		\begin{equation*}
			\liminf_{n\to+\infty}\varphi(x_{n},p)\ge 1-\frac{\eps}{\alpha}.
		\end{equation*}
		By the arbitrary choice of $\eps$, we get $\liminf_{n\to+\infty}\varphi(x_{n},\beta_{k}p)\ge 1$.\\
		We conclude the proof showing that $\limsup_{n\to+\infty}\varphi(x_{n},p)\le 1$. 
		Let $n\in\mathbb{N}$, if $p\in K_{n}$, then $\varphi(x_{n},p)\le 1$, if, instead, $p\not\in K_{n}$, we want to estimate $\varphi(x_{n},p)$ from above. Let $p^{n}\in \partial K_{n}$, such that $dist(p,K_{n})=|p-p^{n}|$, then:
		\begin{equation*}
			\varphi(x_{n},p)\le\varphi(x_{n},p^{n})+\varphi(x_{n},p-p^{n})\le 1+\frac{|p-p^{n}|}{\alpha}\le1+\frac{d(p,K_{n})}{\alpha}.
		\end{equation*}
		The proof is concluded recalling that $d(p,K_{n})\to d(p,K)=0$.
	\end{proof}	
	
	\begin{dfn} \label{defphizero}
		For any $x\in\bO$ and for any $q\in\mathbb{R}^{d}$, we define  $\phizero:\bO\times\mathbb{R}^{n}\to\mathbb{R}$, the \textit{support function} of $K(x)$, as
		\begin{equation*}
			\phizero(x,q):=\sup\left\{p\cdot q \, : \, p\in K(x) \right\}=\sup\left\{p\cdot q \, : \, \varphi(x,p)\le 1 \right\}.
		\end{equation*} 
	\end{dfn}
	\begin{rmk}
		\label{aboutphizero}
		Clearly $\phizero\ge 0$ and $\phizero(x,q)=0$ if and only if $q=0$. Moreover it is a Finsler metric, i.e. a non-negative Borel-measurable function such that (see (1.1)-(1.5) in \cite{DecPal1995} or Definition 1.3 in \cite{GarPonPri2006})
		\begin{enumerate}[(i)]
			\item $\phizero(x,\cdot)$ is positively $1$-homogeneous for every $x\in\bO$;
			\item $\phizero(x,\cdot)$ is convex for a.e. $x\in\bO$;
			\item for every compact subset $B$ of $\Omega$, there exist $\alpha_{B},M_{B}>0$, such that 
			\begin{equation*}
				\alpha_{B}|q|\le\phizero(x,q)\le M_{B}|q|, \ \quad \mbox{for all }(x,q)\in B\times \Rd.
			\end{equation*}
		\end{enumerate}
		The positive $1$-homogeneity and the convexity follows easily from the definition of $\phizero$. The measurability of $\phizero(\cdot,q)$ follows from the continuity of $\phizero(\cdot,q)$ for every $q\in\Rd$(see Proposition \ref{continuityphizero} below) and the measurability of $\varphi(x,\cdot)$ for every $x\in\bO$ is a direct consequence of the convexity, which in particular implies continuity. The last condition is satisfied because of the assumption \eqref{alphaM}. We finally remark that also the functional $\varphi$ satisfies the properties (i)-(iii), where (iii) is satisfied with $M^{-1},\alpha^{-1}$.
	\end{rmk}
	\begin{lmm}
		\label{phizeropuntuale}
		For every $x\in\bO$ and $q\in\Rd$ there exists $p=p(x,q)\in\partial K(x)$ such that 
		\begin{equation*}
			\phizero(x,q)=p\cdot q.
		\end{equation*}
		Moreover, if $K(x)$ is strictly convex $p=p(x,q)$ is unique. 
	\end{lmm}
	\begin{proof}
		The existence of $p=p(x,q)$ follows directly from the definition of $\phizero$ and the compactness of $K(x)$, while the fact that $p\in\partial K(x)$ is a consequence of the convexity.\\
		If $\partial K(x)$ is strictly convex, let us assume by contradiction that there exists $p_{1}\not =p_{2}$ such that $\phizero(x,q)=p_{1}\cdot q=p_{2}\cdot q$. Then, for all $t\in(0,1)$ we have
		\begin{equation*}
			(tp_{1}+(1-t)p_{2})\cdot q=tp_{1}\cdot q+(1-t)p_{2}\cdot q=\phizero(x,q),
		\end{equation*} 
		that means that $tp_{1}+(1-t)p_{2}$ belongs to $\partial K(x)$, contradicting the assumption of strictly convexity. 
	\end{proof}
	
	\begin{prp}
		\label{continuityphizero}
		$\phizero:\bO\times\Rd\to\R$ is continuous. Moreover if $K(x)$ is strictly convex for every $x\in \bO$, also the function $p:\bO\times\Rd\to\Rd$ such that $p(x,q)\cdot q =\phizero(x,q)$ is continuous.
	\end{prp}
	\begin{proof}
		Let $\left((x_{n},q_n)\right)$ converging to $(x,q)$. By Lemma \ref{phizeropuntuale}, there exists a sequence $(p_{n})$, such that 
		\begin{equation*}
			p_{n}\in \partial K(x_{n}) \quad \mbox{and} \quad \phizero(x_{n},q_n)=p_{n}\cdot q.
		\end{equation*}
		Since $K(x)\subset B(0,M)$ for all $x\in\bO$, up to the choice of a subsequence we can assume that there exists $p\in\Rd$ such that $p_{n}\to p$ when $n\to\infty$. Clearly we have that
		$p_{n}\cdot q_n\to p\cdot q$. We want to prove that 
		\begin{equation*}
			\phizero(x,q)=p\cdot q.
		\end{equation*}By continuity of $\varphi$ (Corollary \ref{fullcontinuityphi}) we infer that
		\begin{equation*}
			1=\lim_{n\to\infty}\varphi(x_{n},p_{n})=\varphi(x,p),
		\end{equation*} 
		which implies that $p\in K(x)$ and thus that $p\cdot q \le \phizero(x,q)$.
		We now claim that 
		\begin{equation*}
			p\cdot q \ge w\cdot q, \quad \mbox{for all }w\in \partial K(x).
		\end{equation*}
		Let $w\in \partial K(x)$, by convexity $tw\in K(x)$ for all $0\le t\le 1$, with $tw\in\interior K$, for $t<1$. Then by Proposition \ref{hausdorffdistance} there exists $\nu\in\mathbb{N}$, such that $tw\in K(x_{n})$ for all $n\ge \nu$. Therefore $\phizero(x_{n},q)=p_{n}\cdot q\ge tw\cdot q$ for all $n\ge \nu$, which implies $p\cdot q\ge tw\cdot q$. Letting $t$ to $1$, we get the claim.
		\\Finally if $K(x)$ is strictly convex for every $x$, by the uniqueness of $p=p(x,q)=\argmax \{w\cdot q \, : \, w\in K(x)\}$ proved in Lemma \ref{defphizero} before, we have the convergence of the whole $(p_n)$ to $p$ which implies continuity of $p(x,q)$.
	\end{proof}
	The following is result about which conditions on $\varphi$ insure regularity of $\phizero$. 
	
	\begin{prp}
		\label{differentiabilityphizero}
		If $K(x)$ is strictly convex for all $x\in\Omega$ and $\nabla_x\varphi(x,q)\in C(\Omega\times\Rd)$, then also $\nabla_x\phizero(x,q)\in C(\Omega\times\Rd) $.
	\end{prp}
	\begin{proof}
		We first want to prove that the limit
		\begin{equation}
			\label{diffquotient}
			\lim\limits_{h\to 0}\frac{\phizero(x+he_{i},q)-\phizero(x,q)}{h}
		\end{equation}
		exists and is finite for any $i=1,\dots,d$. We prove it for $i=1$ and we fix $\bar{h}=he_{1}$, in order to simplify the notation. By Lemma\eqref{phizeropuntuale}, there exist unique $p(x+\bar{h},q)=p(x+\bar{h})$ and $p(x,q)=p(x)$, such that
		\begin{align*}
			&\phizero(x+\bar{h},q)=\max\{p\cdot q \ | \ \varphi (x+\bar{h},p)=1 \}= p(x+\bar{h})\cdot q, 
			\\&\phizero(x,q)=\max\{p\cdot q \ | \ \varphi (x,p)=1 \}= p(x)\cdot q.
		\end{align*}
		The limit \eqref{diffquotient} can be rewritten as:
		\begin{equation}
			\label{diffquotient2}
			\lim_{h\to 0}\frac{p(x+\bar{h})\cdot q - p(x)\cdot q}{h}.
		\end{equation}
		Since $\varphi$ is continuously differentiable we can write 
		
		\[\begin{split}
			\varphi(x,p(x+\bar{h}))&=\varphi(x+\bar{h},p(x+\bar{h}))-\partial_{x_{1}}\varphi(x,p(x+\bar{h}))h+o(h)\\& =1-\partial_{x_{1}}\varphi(x,p(x+\bar{h}))h+o(h).\
		\end{split}\]
		
		and
		\[\begin{split}
			\varphi(x+\bar{h},p(x))&=\varphi(x,p(x))+\partial_{x_{1}}\varphi(x,p(x))h+o(h)\\&=1+\partial_{x_{1}}\varphi(x,p(x))h+o(h).
		\end{split}\]
		Which implies that
		\begin{align*}
			&\varphi\left(x,\frac{p(x+\bar{h})}{1-\partial_{x_{1}}\varphi(x,p(x+\bar{h}))h+o(h)}\right)=1 \quad \mbox{and}\\	&\varphi\left(x+\bar{h},\frac{p(x)}{1+\partial_{x_{1}}\varphi(x,p(x))h+o(h)}\right)=1.
		\end{align*}
		and therefore that
		\begin{align*}
			&p(x)\cdot q \ge \frac{p(x+\bar{h})}{1-\partial_{x_{1}}\varphi(x,p(x+\bar{h}))h+o(h)} \cdot q, \quad \mbox{and} \\&p(x+\bar{h})\cdot q \ge \frac{p(x)}{1+\partial_{x_{1}}\varphi(x,p(x))h+o(h)} \cdot q,
		\end{align*}
		Then the quotient in the limit \eqref{diffquotient2} can be estimated from above and below, by:
		\begin{align*}
			&\frac{p(x+\bar{h})\cdot q - p(x)\cdot q}{h}\le\frac{\left(1-\frac{1}{1-\partial_{x_{1}}\varphi(x,p(x+\bar{h}))h+o(h)}\right)p(x+\bar{h})\cdot q}{h}\\
			& = \frac{1}{1-\partial_{x_{1}}\varphi(x,p(x+\bar{h}))h+o(h)} \left(-\partial_{x_{1}}\varphi(x,p(x+\bar{h}))+\frac{o(h)}{h}\right) p(x+\bar{h})\cdot q 
		\end{align*}
		and
		\begin{align*}
			&\frac{p(x+\bar{h})\cdot q - p(x)\cdot q}{h}\ge \frac{\left(\frac{1}{1+\partial_{x_{1}}\varphi(x,p(x))h+o(h)}-1\right)p(x)\cdot q}{h}\\&=\frac{1}{1-\partial_{x_{1}}\varphi(x,p(x))h+o(h)}\left(-\partial_{x_{1}}\varphi(x,p(x))+\frac{o(h)}{h}\right) p(x)\cdot q.
		\end{align*}
		Letting $h\to 0$ in both inequalities, by continuity of $p(x)$ and $\partial_{x_{1}}\varphi$ we get
		\begin{equation*}
			\partial_{x_{1}}\phizero(x,q)=\lim_{h\to 0}\frac{p(x+\bar{h})\cdot q - p(x)\cdot q}{h}=-\partial_{x_{1}}\varphi\left(x,p(x)\right) p(x)\cdot q,
		\end{equation*}
		which shows also that $\partial_{x_{1}}\phizero$ is continuous.
	\end{proof}
	\begin{dfn}
		\label{finslerlength}
		Given the Finsler metric $\phizero$, we can define the associated \textit{Finslerian length functional} $L$:
		\begin{equation*}
			L(\gamma):=\int_{0}^{1}\phizero(\gamma(t),\dot{\gamma}(t))dt, 
		\end{equation*} 
		with $\gamma\in\Winf((0,1))\cap C([0,1])$.
	\end{dfn}
	\begin{rmk}
		\label{indipendenceparam}
		Thanks to the positively $1$-homogeneity of $\phizero$, we have
		\begin{equation*}
			L(\gamma)=\int_{0}^{1}\phizero(\gamma(t),\dot{\gamma}(t))dt=\int_{0}^{l(\gamma)}\phizero(\gamma(s),\dot{\gamma}(s))ds=\int_{a}^{b}\phizero(\gamma(\tau),\dot{{\gamma}}(\tau))d\tau,
		\end{equation*}
		for every orientation preserving parametrization and for every $a,b\in\R$, $a<b$. Therefore we do not lose generality if we restrict to the class of arch length parameterized curves in the definition of the pseudo-distance $d$ associated to $\phizero$. 
	\end{rmk}
	\begin{dfn}\label{distance}
		Let $x,y\in\Omega$, then we define
		\begin{equation}\label{distanceformula}
			d(x,y):=\inf\left\{\int_{0}^{l(\gamma)}\phizero(\gamma(s),\dot{\gamma}(s))ds \ : \ \gamma\in\mathsf{path}(x,y)\right\},
		\end{equation}
		where, if $l(\gamma)$ denotes the Euclidean length of $\gamma$, $\mathsf{path}(x,y)$ is the set
		\begin{equation*}
			\{\gamma\in \Winf((0,l(\gamma)), \Omega)\cap C([0,l(\gamma)],\Omega) \, : \, \gamma(0)=x, \ \gamma(l(\gamma))=y, |\dot{\gamma}(s)|=1 \ \mbox{a.e.}\}.
		\end{equation*}
		The above definition can be extended to all $x,y\in\bO$, in the following way:\\
		\begin{equation}
			\label{dboundary}
			d(x,y):=\inf\left\{\liminf_{n\to\infty}d(x_{n},y_{n}) \ : \ (x_{n}), (y_{n})\in\Omega^{\mathbb{N}} \ \mbox{and } \ x_{n}\to x, \ y_{n}\to y \right\}.
		\end{equation}
	\end{dfn}
	
	\begin{rmk}
		The pseudo-distance defined in Definition \ref{distance} is not a distance in general, since $K(x)$ does not need to be symmetric. However, from now on, we will refer to $d$ as ``distance".
	\end{rmk}
	\begin{rmk}
		\label{triangularineq}
		It is important to point out that, with our definition of $d$, 
		\begin{equation*}
			d(x,y)\le d(x,z)+d(z,y)
		\end{equation*}
		for every $x,y\in\bO$ and $z\in\Omega$, but, as shown in the Example \ref{notriang} below, it may fail if the interpolating point $z$ belongs to the boundary. However, as proved in Proposition \ref{continuityd} if we assume $\pO$ to be Lipschitz the triangular inequality holds for every $x,y,z\in \bO$.
	\end{rmk}
	\begin{xpl}
		\label{notriang}
		Let $\Omega\subset\R^{2}$, defined as $\Omega:=B_{1}(0)\setminus\{(0,y):0< y<1\}$ and  $K(x)\equiv B_{1}(0)$ for all $x\in \Omega$. Then $\phizero(x,q)=|q|$ for any $q\in\R^{d}$ and for any $x\in\Omega$. Let's take $(x_{1},x_{2})=\left(\frac{1}{2},\frac{1}{2}\right)$, $(y_{1},y_{2})=\left(-\frac{1}{2},\frac{1}{2}\right)$ and $(z_{1},z_{2})=\left(0,\frac{1}{2}\right)$ then 
		\begin{equation*}
			d((x_{1},x_{2}),(y_{1},y_{2}))=\inf\left\{\int_{0}^{1}|\dot{\gamma}(t)|dt: \gamma\in\mathsf{path}((x_{1},x_{2}),(y_{1},y_{2})) \right\}=\frac{\sqrt{2}}{2}+\frac{\sqrt{2}}{2}=\sqrt{2},
		\end{equation*} 
		while
		\begin{equation*}
			d((x_{1},x_{2}),(z_{1},z_{2}))+d((z_{1},z_{2}),(y_{1},y_{2}))=\frac{1}{2}+\frac{1}{2}=1.
		\end{equation*}
	\end{xpl}
	
	\begin{prp}
		\label{continuityd}
		The function distance $d:\bO\times\bO\to\R$ is lower semicontinuous, continuous in $\Omega\times\Omega$ and
		\begin{equation}
			\label{bilipdistance}
			\alpha |x-y|_{\Omega}\le d(x,y)\le M|x-y|_{\Omega}, \quad \mbox{for all }x,y\in\Omega, 
		\end{equation} 
		where
		\begin{equation*}
			|x-y|_{\Omega}:=\inf\left\{\int_{0}^{1}|\dot{\gamma}(t)|dt: \gamma\in\mathsf{path}(x,y) \right\},
		\end{equation*}
		Moreover if $\pO$ is Lipschitz, $d$ is continuous and equivalent to the Euclidean distance. In particular, the triangular inequality holds for every $x,y,z\in\bO$.
	\end{prp}
	\begin{proof}
		The proof of the \eqref{bilipdistance} follows from the property \eqref{alphaM} of $K$. Indeed, by that property we infer that 
		\begin{equation*}
			\alpha|q|\le\phizero(x,q)\le M|q|, \quad \mbox{for every }x\in\Omega,
		\end{equation*}
		which implies
		\begin{equation}\label{prebilipschitzdistance}
			\alpha \int_{0}^{l(\gamma)}|\dot{\gamma}(t)|dt \le \int_{0}^{l(\gamma)}\phizero (\gamma(t),\dot{\gamma}(t)) dt\le M \int_{0}^{l(\gamma)}|\dot{\gamma}(t)|dt, \quad \mbox{for every } \gamma\in\mathsf{path}(x,y),
		\end{equation}
		that proves the \eqref{bilipdistance}.
		Clearly the \eqref{bilipdistance} implies the continuity of $d$ in $\Omega\times\Omega$. The lower semicontinuity in $\bO\times\bO$ follows from the \eqref{dboundary}, the extended definition of $d$ up to the boudary. Finally, if $\pO $ is lipschitz, there exists $C_{\Omega}>0$, such that 
		\begin{equation}
			\label{geoeucdistance}
			|x-y|\le |x-y|_{\Omega}\le C_{\Omega}|x-y|.
		\end{equation}
		The \eqref{geoeucdistance} allows us to extend $d$ continuously to the boundary, defining, for every $x,y\in\pO$:
		\begin{equation*}
			d(x,y):=\lim_{n\to\infty}d(x_{n},y_{n}), \quad \mbox{for some }x_{n}\to x, y_{n}\to y.
		\end{equation*}
		This definition is well posed and equivalent to the one given by \eqref{dboundary}, indeed, if $(x_{n}),(x'_{n})\subset\Omega$ are two sequences converging to $x$ and $(y_{n}),(y'_{n})\subset\Omega$ are two sequences converging to $y$, by using the triangular inequality for points in $\Omega$ it holds
		\begin{align*}
			&\left|d(x_{n},y_{n})-d(x'_{n},y'_{n})\right|=\left|d(x_{n},y_{n})-d(x'_{n},y_{n})+d(x'_{n},y_{n})-d(x'_{n},y'_{n})\right|\le\\ & \le \left|d(x_{n},x'_{n})+d(y'_{n},y_{n})\right|\le C_{\Omega}|x_{n}-x'_{n}|+C_{\Omega}|y_{n}-y'_{n}|\to 0.
		\end{align*}
	\end{proof}
	Keeping in mind the definition of $\mathsf{path}$, for all $x,y\in\bO$ we define the set $\pathbar(x,y)$, as
	\begin{align*}
		\{\gamma\in \Winf((0,l(\gamma)), \bO)\cap C([0,l(\gamma)],\bO) \ : \ 
		\gamma(0)=x, \ \gamma(l(\gamma))=y, |\dot{\gamma}(s)|=1 \ \mbox{a.e.}\}.
	\end{align*}
	
	\begin{prp}
		\label{existencecurve}
		Let $x,y\in \Omega$. Then there exists $\gamma \in \overline{\mathsf{path}}(x,y)$, 
		such that 
		\begin{equation*}
			d(x,y)=\int_{0}^{l(\gamma)}\phizero(\gamma(s), \dot{\gamma}(s))ds.
		\end{equation*}
		We say that such a $\gamma$ is a geodesic for $d$ connecting $x$ to $y$.
	\end{prp}
	\begin{proof} 
		We first observe that since $\Omega$ is a bounded and connect set, $|x-y|_{\Omega}<+\infty$ and by \eqref{bilipdistance} we also have $d(x,y)<+\infty$.
		We then consider a sequence $(\gamma_{n})_{n\in\mathbb{N}}\subset\mathsf{path}(x,y)$ converging to the infimum, i.e. such that
		\begin{equation}
			\lim_{n\to\infty}\int_{0}^{l(\gamma_{n})}\phizero(\gamma_{n}(s),\dot{\gamma_{n}}(s))ds=d(x,y).
		\end{equation}
		Thanks to the \eqref{prebilipschitzdistance} can also assume that $|l(\gamma_{n})|\le L$, for all $n$, for some $L\ge 0$. Then there exists a subsequence of $(\gamma_{n})$, that we will still call $(\gamma_{n})$, such that
		\begin{equation}
			\label{limitlengths}
			\lim_{n\to\infty}l(\gamma_{n})=\bar{l}.
		\end{equation}
		Moreover it holds
		\begin{itemize}
			\label{boundness}\item $|\gamma_{n}(s)|\le C$, for some $C\ge 0$, since $\Omega$ is bounded;
			\label{equicont}\item $|\dot{\gamma_{n}}(s)|=1$ a.e. in $[0,l(\gamma_{n})]$, by definition of $\mathsf{path}(x,y)$.
		\end{itemize}
		For every $\eps>0$, by Arzelà-Ascoli Theorem, there exists a subsequence converging uniformly to some curve $\gamma^{\eps}$ in $[0,\bar{l}-\eps]$. By a diagonal argument, we can again extract a subsequence, that we still call $(\gamma_{n})$, and a curve $\gamma:[0,\bar{l}]\to\bar{\Omega}$ such that 
		\begin{align}
			\notag&\gamma_{n} \ \mbox{converges punctually to }\gamma \ \mbox{in } [0,\bar{l}), \\ \notag&\lim_{n\to\infty}l(\gamma_{n})=\bar{l}, \\
			\label{norminfinitygamma}& |\dot{\gamma}(s)|=1 \ \mbox{a.e. in } [0,\bar{l})
		\end{align}
		Moreover, by Lebesgue Theorem we have convergence in $L^{2}((0,\bar{l}))$ and then, by the stability of weak derivatives, we have that $\dot{\gamma_{n}}\rightharpoonup \dot{\gamma}$ in $L^{2}((0,\bar{l}))$.
		By \eqref{norminfinitygamma} we infer that $\lim_{n\to\infty}||\dot{\gamma}_{n}||_{L^2}=1=||\dot{\gamma}||_{L^2}$ and thus $\dot{\gamma_{n}}\rightarrow\dot{\gamma}$ in $L^{2}((0,\bar{l}),\bar{\Omega})$ and a.e. in $[0,\bar{l})$. Continuity of $\phizero$ (see Proposition \ref{continuityphizero}) then implies that $\phizero(\gamma_{n}(s), \dot{\gamma_{n}}(s))$ converges a.e. to $\phizero(\gamma(s),\dot\gamma(s))$ and, according to the Lebesgue's Theorem:
		\begin{align}
			\notag&d(x,y)=\lim_{n\to\infty}\int_{0}^{l(\gamma_{n})}\phizero(\gamma_{n}(s),\dot{\gamma_{n}}(s))ds=
			\label{lebesgue}\\&=\int_{0}^{\bar{l}}\phizero(\gamma(s), \dot{\gamma}(s))ds.
		\end{align} 
		Finally we prove that
		\begin{equation*}
			\lim_{\eps\to 0}\gamma(\bar{l}-\eps)=y.
		\end{equation*}
		Let $\eps>0$, then we choose $n$ big enough so that $|l(\gamma_{n})-\bar{l}|<\eps$ and $\left|\gamma(\bar{l}-\eps)-\gamma_{n}(\bar{l}-\eps)\right|<\eps$, then
		\begin{equation*}
			\left|\gamma(\bar{l}-\eps)-y\right|\le\left|\gamma(\bar{l}-\eps)-\gamma_{n}(\bar{l}-\eps)\right|+\left|\gamma_{n}(\bar{l}-\eps)-\gamma_{n}(l(\gamma_{n}))\right|<\eps +|\bar{l}-\eps-l(\gamma_{n})|<3\eps.
		\end{equation*}
		Then, defining $\gamma(\bar{l}):=y$, we have that $\gamma \in\overline{\mathsf{path}}(x,y)$.
	\end{proof}
	\begin{prp}
		\label{existencecurvebondary}
		Let $x,y\in \bO$ such that $d(x,y)<+\infty$. Then there exists $\gamma \in \overline{\mathsf{path}}(x,y)$, such that 
		\begin{equation*}
			d(x,y)=\int_{0}^{l(\gamma)}\phizero(\gamma(s), \dot{\gamma}(s))ds.
		\end{equation*}
		We say that such a $\gamma$ is a geodesic for $d$ connecting $x$ to $y$.
	\end{prp}
	\begin{proof}
		If $x,y\in \Omega$, then the existence of such a curve has already been proven in Proposition \ref{existencecurve}. Let $x,y\in\pO$ such that $d(x,y)<+\infty$ and let $(x_{n}),(y_{n})$ two sequences (they can be constructed for instance by a diagonal argument) converging respectively to $x$ and $y$ such that
		\begin{equation*}
			d(x,y)=\lim_{n\to+\infty}d(x_{n},y_{n}).
		\end{equation*}
		For every $n$ we consider $\gamma_{n}\in\overline{\mathsf{path}}(x_{n},y_{n})$ such that 
		\begin{equation*}
			d(x_{n},y_{n})=\int_{0}^{l(\gamma_{n})}\phizero(\gamma_{n}(s),\dot\gamma_{n}(s))ds,
		\end{equation*}
		obtaining that
		\begin{equation}
			\label{distboundary2}
			d(x,y)=\lim_{n\to\infty}\int_{0}^{l(\gamma_{n})}\phizero(\gamma_{n}(s),\dot{\gamma}_{n}(s))ds.
		\end{equation}
		Finally we apply the same reasoning as in the proof of Proposition \ref{existencecurve} and, up to the choice of subsequences, we get a curve $\gamma\in\overline{\mathsf{path}}(x,y)$, such that 
		\begin{equation*}
			\lim_{n\to\infty}\int_{0}^{l(\gamma_n)}\phizero(\gamma_{n}(s),\dot\gamma_{n}(s))ds=\int_{0}^{l(\gamma)}\phizero(\gamma(s),\dot\gamma(s))ds.
		\end{equation*}
	\end{proof}
	
	\begin{rmk}
		\label{gammainclosurepath}
		From Proposition \ref{existencecurve} and Proposition \ref{existencecurvebondary} we infer that, for every $x,y\in\bO$ such that $d(x,y)<+\infty$, the optimal $\gamma\in\pathbar(x,y)$
		\begin{equation*}
			d(x,y)=\inf \left\{\int_{0}^{\l(\gamma)}\phizero(\gamma(s),\dot{\gamma}(s))ds \ : \gamma\in{\mathsf{path}(x,y)}   \right\},
		\end{equation*}
		
		We point out that, as shown in the example below, it can be 
		\[d(x,y)>\inf \left\{\int_{0}^{\l(\gamma)}\phizero(\gamma(s),\dot{\gamma}(s))ds \ : \gamma\in{\overline{\mathsf{path}}(x,y)}   \right\}. \]
	\end{rmk}
	\begin{xpl}
		\label{exnotlip}
		Let  $\Omega$, $K(x)$, $(x_{1},x_{2})$ and $(y_{1},y_{2})$ be defined as in the Example \ref{notriang}. Then 
		\begin{equation*}
			d((x_{1},x_{2}),(y_{1},y_{2}))=\inf\left\{\int_{0}^{l(\gamma)}|\dot{\gamma}(s)|ds: \gamma\in\mathsf{path}((x_{1},x_{2}),(y_{1},y_{2})) \right\}=\frac{\sqrt{2}}{2}+\frac{\sqrt{2}}{2}=\sqrt{2},
		\end{equation*} 
		while
		\begin{equation*}
			\inf\left\{\int_{0}^{l(\gamma)}|\dot{\gamma}(s)|ds:\gamma\in\overline{\mathsf{path}}((x_{1},x_{2}),(y_{1},y_{2}))\right\}=1.
		\end{equation*}
	\end{xpl}

	\begin{prp}
		\label{omegalipschitz}
		If $\Omega$ has a Lipschitz boundary then, for every $x,y\in\bO$, $d(x,y)<+\infty$  and
		\begin{equation}
			\label{distanceinastarshaped}
			d(x,y)=\inf\left\{\int_{0}^{l(\gamma)}\phizero(\gamma(s),\dot\gamma(s))ds\ : \ \gamma\in\overline{\mathsf{path}}(x,y) \right\}.
		\end{equation}
	\end{prp}
	\begin{proof}
		
		The first part of the statement is a consequence of the second part of Proposition \ref{continuityd} while the second part follows directly by Remark \ref{gammainclosurepath} and Lemma. \ref{approxlipsets}.
	\end{proof}

	\section{The maximal and minimal extension and some structure results }
	\label{section3}
	From now on we will assume $\partial\Omega$ to be lipschitz. \footnote{All the results present in this section can be proven also for a more general bounded and connected open set $\Omega$ with some cautions when we consider points on the boundary. However this is beyond the scope of this paper and we prefer to lighten the proofs, considering a more simple setting.}We also assume that the boundary datum $g$ of \eqref{problem} satisfies
	\begin{equation*}
		g(y_{2})-g(y_{1})\le d(y_{1},y_{2}) \quad \mbox{for all } y_{1},y_{2}\in\pO.
	\end{equation*} 
	We will say (even if is not exactly the standard definition) that $g$ is $1$-lipschitz w.r.t. $d$.
	In this section we consider the maximal and minimal extensions defined respectively by 
	\begin{align}
		\label{S+}&S^{-}(x)=\sup\left\{g(y)-d(x,y): y\in\pO\right\},\\
		\label {S-}&S^{+}(x)=\inf\left\{g(y)+d(y,x): y\in\pO\right\},
	\end{align}
	\begin{prp}
		The following facts hold true:
		\label{maxminextensioncurves}
		\begin{enumerate}
			\item $S^{+}(y)=S^{-}(y)=g(y)$ for every $y\in\pO$;
			\item for every $x\in\bO$ there exist $y_{1},y_{2}\in\pO$ and $\gamma_{1}\in\overline{\mathsf{path}}(y_{1},x),\gamma_{2}\in\overline{\mathsf{path}}(x,y_{2})$ with
			\begin{equation*}
				S^{+}(x)=g(y_1)+\int_{0}^{l(\gamma_1)}\phizero(\gamma_1(s),\dot{\gamma}_1(s))ds,
			\end{equation*}
			such that $\gamma_1(s)\in \Omega$ for every $s\in(0,\ell(\gamma_1)]$ and 
			\begin{equation*}
				S^{-}(x)=g(y_2)-\int_{0}^{l(\gamma_2)}\phizero(\gamma_2(s),\dot{\gamma}_2(s))ds,
			\end{equation*}
			such that $\gamma_2(s)\in\Omega$ for every $s\in(0,\ell(\gamma_2)]$. We will refer to such $y_1,\gamma_1$ and $y_2,\gamma_2$ as optimal respectively for $S^+(x)$ and $S^-(x)$;
			\item if $y_1\in\pO$ and $\gamma_1\in\pathbar(y_1,x)$ are optimal for $S^+(x)$, then for every $\hat{s}\in(0,l(\gamma_1))$, with $z_1=\gamma_1(\hat{s})$, we have 
			\begin{equation}
				\label{S+z}
				S^{+}(z)=g(y_{1})+\int_{0}^{\hat{s}}\phizero(\gamma_{1}(s),\dot\gamma_{1}(s))ds,
			\end{equation}
			and if $y_2\in\pO$ and $\gamma_2\in\pathbar(y_2,x)$ are optimal for $S^-(x)$, then for every $\hat{s}\in(0,l(\gamma_2))$, with $z_2=\gamma_2(\hat{s})$,
			\begin{equation}
				\label{S-z}
				S^{-}(z)=g(y_{2})+\int_{0}^{\hat{s}}\phizero(\gamma_{2}(s),\dot\gamma_{2}(s))ds.
			\end{equation}
		\end{enumerate}
		
	\end{prp}
	
	\begin{proof}
		The proof of (1) follows directly from the definition of $S^+$ and $S^-$ and by $g$ being $1$-lipschitz.\\
		We proceed with the proof of (2).
		We first claim that for any $x\in\bO$ there exist $y\in\pO$ and $\gamma\in\pathbar(x,y)$, such that 
		\begin{equation}\label{optimalgamma}
			S^{+}(x)=g(y)+\int_{0}^{l(\gamma)}\phizero(\gamma(s),\dot{\gamma}(s))ds.
		\end{equation} 
		If $x\in\pO$ there is nothing to prove because of the point (1). 
		If $x\in \Omega$, since $g$ is continuous, $d$ is lower semicontinuous (see Proposition \ref{continuityd}) and $\bO$ is a compact set, there exists $y\in\pO$ such that 
		\begin{equation*}
			S^{+}(x)=g(y)+d(y,x).
		\end{equation*}
		Now, the existence of an optimal curve is given by Proposition \ref{existencecurvebondary} and \eqref{optimalgamma} is proven.
		We then show that it is possible to find a curve that solves \eqref{optimalgamma} all contained in $\Omega$ except from the first point.
		Let $\hat{y}=\gamma (\hat{s})$ for some $\hat{s}\in(0,1]$, be the last point of the boundary touched by $\gamma$, then $\gamma:[\hat{s},l(\gamma)]\to\bO$ touches the boundary only in $\gamma(\hat{s})$. We show that 
		\begin{equation*}
			g(y)+\int_{0}^{l(\gamma)}\phizero(\gamma(s),\dot{\gamma}(s))ds=g(\hat{y})+\int_{\hat{s}}^{l(\gamma)}\phizero(\gamma(s),\dot{{\gamma}}(s))ds.
		\end{equation*}
		By definition of $S^{+}$ we know that ``$\le$" holds.
		Let us assume by contradiction that the inequality is strict. Then 
		\begin{equation}
			\label{contradiction}
			d(y,\hat{y})\le\int_{0}^{\hat{s}}\phizero(\gamma(s),\dot{\gamma}(s))ds<g(\hat{y})-g(y),
		\end{equation}
		which is in contradiction with the $1$-lipschitzianity of $g$ w.r.t. $d$.\\
		We conclude showing (3). Let $y\in\pO$ and $\gamma$ be optimal for the definition of $S^+(x)$ and let $z=\gamma(\hat{s})$, for some $\hat{s}\in(0,l(\gamma))$.
		We prove that
		\begin{equation*}
			S^{+}(z)=g(y_{1})+\int_{0}^{\hat{s}}\phizero(\gamma(s),\dot{\gamma}(s))ds=g(y)+d(y,z).
		\end{equation*}
		Assume by contradiction that there exist $\hat{y}\in\pO$ and $\eta\in\overline{\mathsf{path}}(\hat{y},z)$ such that
		\begin{equation}
			\label{contradictioncurve}
			g(\hat{y})+\int_{0}^{l(\eta)}\phizero(\eta(s),\dot\eta(s))ds<g(y)+\int_{0}^{\hat{s}}\phizero(\gamma(s),\dot\gamma(s))ds.
		\end{equation}
		Then if we call $\hat{\gamma}$ the piece of $\gamma$ connecting $z$ to $x$ and $\hat{\eta }$ the reparametrized version of the curve obtained by gluing  $\eta$ and $\hat{\gamma}$, we have
		\begin{equation*}
			g(\hat{y})+\int_{0}^{l(\hat{\eta})}\phizero(\hat{\eta}(s),\dot\hat{\eta}(s))ds<g(y)+\int_{0}^{l(\gamma)}\phizero(\gamma(s),\dot{\gamma}(s))ds, 
		\end{equation*}
		contradicting the optimality of $y$ and $\gamma$.\\
		The proof of (2) and (3) for $S^-$ is analogous. 
		
	\end{proof} 
	\begin{prp}
		\label{S+S-lipschitz}
		The maximal and minimal extensions, defined by \eqref{S+} and \eqref{S-} are $1$-Lipschitz w.r.t. to $d$, i.e.
		\begin{equation*}
			\Sp(y)-\Sp(x)\le d(x,y) \ \mbox{and }
			\Sm(y)-\Sm(x)\le d(x,y) \quad \mbox{for all } x,y\in\bO.
		\end{equation*} 
		Moreover  $\Sp, \Sm\in\Winf(\Omega)\cap C(\bO)$.
	\end{prp}
	\begin{proof}
		Let $x_1\in \partial\Omega$ such that $S^+(x)=g(x_1)+d(x_1,x)$ then for every $x,y\in\bO$,
		\begin{equation*}
			\Sp(y)-\Sp(x)\le g(x_{1})+d(x_{1},y)-g(x_{1})-d(x_{1},x)\le d(x,y),
		\end{equation*}
		where the last inequality follows from the triangular inequality. Moreover since $\pO$ is lipschitz, by Proposition \ref{continuityd}, we have that for every $x,y\in\bO$
		\begin{equation*}
			|\Sp(y)-\Sp(x)|\le \max\{d(x,y),d(y,x)\}\le C_{\Omega}|x-y|,
		\end{equation*}
		for some constant $C_{\Omega}$ depending on $\Omega$.
	\end{proof}
	
	\begin{prp}
		\label{1lipimpliessolution}
		Let $u:\bO\to\R$ be such that 
		\begin{equation*}
			u(y)-u(x)\le d(x,y) \quad \mbox{for all }x,y\in\bO.
		\end{equation*}
		Then $u\in W^{1,\infty}(\Omega)\cap C(\bO)$ and $\gradu(x)\in K(x)$ for a.e. $x\in\Omega$ (in particular $\nabla u(x)\in K(x)$ for any point $x$ at which $u$ is differentiable ).
	\end{prp}
	\begin{proof}
		For proving that  of $u\in W^{1,\infty}(\Omega)\cap(\bO)$, one can apply the same reasoning of the one used for $S^+$ in the proof of Proposition \ref{S+S-lipschitz} above.\\
		By Proposition \ref{dualitysuppfunc} we know that
		\begin{equation}
			\label{convexduality}
			K(x)=\{p\in\Rd \, : \, p\cdot q \le 1, \ \text{for all} \ q \ \text{s.t.} \ \phizero(x,q)\le 1 \}.
		\end{equation}	
		We will prove that $\nabla u(x)$ is contained the set defined in the right hand side of the equality \eqref{convexduality} for a.e. $x\in\Omega$.
		Let $x$ be a point of differentiability for $u$ and $q\in\Rd$ such that $\phizero(x,q)\le 1$. Then 
		\begin{equation*}
			\gradu(x)\cdot q=\lim_{h\to 0}\frac{u(x+hq)-u(x)}{h}\le\lim_{h\to 0} \frac{d(x,x+hq)}{h},
		\end{equation*}
		where the inequality holds by assumption. By definition of $d$ and $1$-homogeneity of $\phizero$, we get 
		\begin{equation}
			\gradu(x)\cdot q\le \lim \frac{1}{h}\int_{0}^{1}\phizero(x+thq,hq)=\phizero(x,q)\le 1.
		\end{equation}
		By \eqref{convexduality}, we have $\gradu(x)\in K(x)$.
	\end{proof}
	\begin{crl}
		\label{S+S-gradient}
		$S^{+},S^{-}$ are solutions of \eqref{problem}.
	\end{crl}
	\begin{proof}
		The proof follows directly from Proposition \ref{maxminextensioncurves} (i), Proposition \ref{S+S-lipschitz} and Proposition \ref{1lipimpliessolution}.
	\end{proof}
	
	\begin{prp}
		\label{solutions}
		If $u$ is a solution of \eqref{problem}, then
		\begin{equation*}
			S^{-}(x)\le u(x) \le S^{+}(x), \quad \text{for every }x\in\bO.
		\end{equation*}
	\end{prp}
	\begin{proof}
		Let us prove that $u(x)\le S^{+}(x)$ for all $x\in\Omega$. The case $S^{-}(x)\le u(x)$ is similar. If we knew that $u$ is $1$-lipschitz w.r.t. $d$, then the thesis would be straightforward. Indeed we would have
		\begin{equation*}
			u(x)-g(y)=u(x)-u(y)\le d(y,x) \quad \mbox{for every } y\in\pO,
		\end{equation*}
		that implies $u(x)\le\inf \{g(y)+d(y,x)\}$.\\
		The following proposition will show that being a solution of \eqref{problem} implies $1$-lipschitzianity w.r.t. to $d$.
	\end{proof}
	\begin{prp}
		\label{1lipu}
		Let $u\in\Winf(\Omega)\cap C(\bO)$ such that $\nabla u(x)\in K(x)$ a.e. in $\Omega$, then
		\begin{equation*}
			u(y)-u(x)\le d(x,y) \quad \mbox{for all } x,y\in\bO.
		\end{equation*}  
	\end{prp}
	
	\begin{proof}
		
		Let $x,y\in\bO$ and $N:=\{x\in\Omega \, : \, u \ \mbox{is not differentiable at }x \ \mbox{or }\nabla u(x)\not\in K(x) \}$. It then holds that 
		\begin{equation*}
			u(y)-u(x)\le\inf\left\{\int_{0}^{l(\gamma)}\phizero(\gamma(s),\dot{\gamma}(s))ds:\gamma\in\mathsf{path}(x,y) \ \mbox{and} \ \gamma \ \mbox{transversal to} \ N \right\},
		\end{equation*}
		where transversal means that $\mathcal{H}^{1}(\gamma((0,1))\cap N)=0$. Indeed, for any $\gamma$ transversal to $N$, we have 
		\begin{equation*}
			u(y)-u(x)=\int_{0}^{l(\gamma)}(u\circ\gamma)'ds=\int_{0}^{l(\gamma)}\nabla u(\gamma(s))\cdot\dot{\gamma}(s)ds\le\int_{0}^{l(\gamma)}\phizero(\gamma(s),\dot{\gamma}(s))ds,
		\end{equation*}
		where the inequality follows from the definition of $\phizero$, since $\nabla u(\gamma(s))\in K(\gamma(s))$. The proof is concluded applying Lemma \ref{approximationbytrasversalcurves} in the Appendix, thanks to which we know that is possible to approximate any curve $\gamma\in\mathsf{path}(x,y)$ with a sequence of $(\gamma_n)\in\mathsf{path}(x,y)$ transversal to $N$ for any $n$.
		
	\end{proof}
	\begin{rmk}\label{lipurmk}
		As a direct consequence of the Proposition \ref{1lipu} above we have that, if $u$ is a solution of \eqref{problem}, then $g$ is $1$-lipschitz w.r.t. $d$.
	\end{rmk}
	\begin{rmk}
		\label{equivalenced}
		Another consequence of the results presented above is that $$d(x,y)=\delta_{\phizero}(x,y)$$ for every $x,y\in\bO$, where
		\begin{align}
			\label{dphi}
			\notag \delta_{\phizero}(x,y)&=\sup\left\{u(y)-u(x): u\in\Winf(\Omega)\cap C(\bO), \ \esup_{x\in\Omega}\varphi^{00}(x,\nabla u(x))\le 1 \right\}\\
			&=\sup\left\{u(y)-u(x): u\in\Winf(\Omega)\cap C(\bO), \ \esup_{x\in\Omega}\varphi(x,\nabla u(x))\le 1 \right\},
		\end{align}
		which is a natural way to define a distance starting from a Finsler metric (see \cite{DecPal1995}, \cite{GarPonPri2006}).
	\end{rmk}
	\begin{proof}
		The inequality $``\ge"$ is trivial using Lemma \ref{1lipu} and the definition of $\varphi$.\\
		To recover the converse inequality we consider $u(\cdot)=d(z,\cdot)$, which is $1-$lipschitz thanks to the triangular inequality. 
	\end{proof}
	\begin{thm}
		\label{curveofcoincidence}
		Let $x\in \Omega$ such that $S^{-}(x)=S^{+}(x)$. Then for every $y_{1},y_{2}\in\pO$ and $\gamma_{1}\in\overline{\mathsf{path}}(y_{1},x),\gamma_{2}\in\overline{\mathsf{path}}(x,y_{2})$ that are optimal for $S^{+}(x)$ and $S^{-}(x)$ in the sense of Proposition \ref{maxminextensioncurves}, if one defines the curve
		\begin{equation}\label{geodesicgamma}
			\gamma(t):= \begin{cases}
				\gamma_{1}(t) \ &\mbox{for }  0\le t \le l(\gamma_1)\\
				\gamma_{2}(t-l(\gamma_{1})) \ &\mbox{for } l(\gamma_1)< t\le l(\gamma_2)+l(\gamma_1)
			\end{cases},
		\end{equation}
		it holds that $\gamma\in\overline{\mathsf{path}}(y_{1},y_{2})$ and $\gamma$ is a geodesic for $d(y_1,y_2)$.
		\\ For every $z\in\gamma$ and for every solution $u$ of \eqref{problem},  it also holds
		\begin{equation}
			\label{s+eqs-z}
			S^{-}(z)=u(z)=S^{+}(z).
		\end{equation}
		Finally, $S^{+},S^{-}$ and every solution $u$ of \eqref{problem} are derivable along $\gamma$ for $\mathcal{H}^{1}$ a.e. point of $\gamma$, and $\nabla_{\gamma} S^{+}=\nabla_{\gamma} S^{-}=\nabla_{\gamma} u=\phizero(\gamma,\dot{\gamma})$.
	\end{thm}
	\begin{proof}
		Let $x$, $y_1$, $y_2$, $\gamma_1$ and $\gamma_2$ be as in the assumptions. Then $S^-(x)=S^+(x)$ can be rewritten as
		\begin{equation}
			\label{s+eqs-2}
			g(y_{1})+\int_{0}^{l(\gamma_{1})}\phizero(\gamma_{1}(s),\dot{\gamma}_{1}(s))dt=g(y_{2})-\int_{0}^{l(\gamma_{2})}\phizero(\gamma_{2}(s),\dot{\gamma}_{2}(s))ds,
		\end{equation}
		which implies 
		\begin{align*}
			&d(y_{1},y_{2})\ge g(y_{2})-g(y_{1})=\\
			&=\int_{0}^{l(\gamma_{1})}\phizero(\gamma_{1}(s),\dot{\gamma}_{1}(s))ds+\int_{0}^{l(\gamma_{2})}\phizero(\gamma_{2}(s),\dot{\gamma}_{2}(s))ds=\\&=\int_{0}^{l(\gamma)}\phizero(\gamma(s),\dot{\gamma}(s))ds\ge d(y_{1},y_{2}),
		\end{align*}
		proving that $\gamma$ is a geodesic connecting $y_1$ to $y_2$.\\
		For every $z\in\gamma_{2}$, it holds:
		\begin{equation}
			\label{mugamma2}
			S^{-}(z)-S^{-}(x)=g(y_{2})-d(z,y_{2})-g(y_{2})+d(x,y_{2})=d(x,z),
		\end{equation}
		where the first equality comes from (3) of Proposition \ref{maxminextensioncurves}. Analogously, if $z\in\gamma_{1}$, $S^{-}(x)-S^{-}(z)=d(z,x)$.\\
		Let now $z_{1}\in\gamma_{1}$ and $z_{2}\in\gamma_{2}$, then:  .
		\begin{align*}
			&d(z_{1},z_{2})\ge S^{-}(z_{2})-S^{-}(z_{1})\ge S^{-}(z_{2})-S^{+}(z_{1})=
			\\&=S^{-}(z_{2})-S^{-}(x)+S^{+}(x)-S^{+}(z_{1})=d(x,z_{2})+d(z_{1},x)\ge d(z_{1},z_{2}),
		\end{align*}
		where the last equality from what we have just proved above. This proves that $S^+(z_1)=S^-(z_1)$, switching the role of $z_1$ and $z_2$ we then have that $S^{+}$ and $S^{-}$ coincides along $\gamma$. Moreover, if we call $S_{1}$ the length of $\gamma$ from $z_{1}$ to $y_{2}$ and $S_{2}$ the length of $\gamma$ from $y_{1}$ to $z_{2}$, it holds
		\begin{align}
			\label{Sz1}&S^{-}(z_{1})=g(y_{2})-\int_{0}^{S_{1}}\phizero(\gamma(s),\dot{\gamma}(s))ds \quad \mbox{and}
			\\\label{Sz2}&S^{+}(z_{2})=g(y_{1})+\int_{0}^{S_{2}}\phizero(\gamma(s),\dot{\gamma}(s))ds.
		\end{align}
		Let us prove the \eqref{Sz1}, the \eqref{Sz2} can be proven similarly. By the definition of $S^-(z_1)$ it is sufficient to prove that 
		\begin{equation}
			\label{claim}
			g(y_{2})-\int_{0}^{S_{1}}\phizero(\gamma(s),\dot{\gamma}(s))ds\ge S^{-}(z_{1}).
		\end{equation}
		Since we know that 
		\begin{equation*}
			S^-(z_1)=S^{+}(z_{1})=g(y_{1})+\int_{S_{1}}^{l(\gamma)}\phizero(\gamma(s),\dot{\gamma}(s))ds,
		\end{equation*}
		where the last equality holds again by (3) of Proposition \ref{maxminextensioncurves}, the inequality \eqref{claim} holds if and only if
		\begin{equation*}
			g(y_{2})-g(y_{1})\ge\int_{0}^{l(\gamma)}\phizero(\gamma(s),\dot{\gamma}(s))ds,
		\end{equation*}
		that is true by \eqref{s+eqs-2}. We point out that this result extends at all the points of $\gamma$ what we proved in (3) of Proposition \ref{maxminextensioncurves} for the points of  $\gamma_{1}$ and $\gamma_{2}$. We also recall that, if $u$ is a solution of \eqref{problem} then, by Proposition \ref{solutions}, $u=S^{+}=S^{-}$ along $\gamma$.
		\\ We conclude showing that the derivative along $\gamma$ of $S^{+}$ and $S^{-}$ (and also of any solution $u$ of \eqref{problem}, by Proposition \ref{solutions}) exists in every point at which the curve is differentiable and it is equal to $\phizero$. Let $\bar{s}\in[0,l(\gamma)]$ a Lebesgue point for $\phizero(\gamma(s),\dot{\gamma}(s))$, such that the curve is differentiable at $\bar{s}$, then for all what we proved above we can write
		\begin{align*}
			&\lim_{h\to 0}\frac{S^{+}(\gamma(\bar{s}+h))-S^{+}(\gamma(\bar{s}))}{h}=
			\\&=\lim_{h\to 0}\frac{1}{h}\left(g(y_{1})+\int_{0}^{\bar{s}+h}\phizero(\gamma(s),\dot{\gamma}(s))ds-g(y_{1})-\int_{0}^{\bar{s}}\phizero(\gamma(s),\dot{\gamma}(s))ds\right).
		\end{align*}
		Which is equal to 
		\begin{equation*}
			\lim_{h\to 0}\frac{1}{h}\int_{\bar{s}}^{\bar{s}+h}\phizero(\gamma(s),\dot{\gamma}(s))ds=\phizero(\gamma(\bar{s}),\dot{\gamma}(\bar{s})).
		\end{equation*}
		
	\end{proof}
	\begin{dfn}
		We will refer to the set of points
		\begin{equation*}
			\mathcal{U}:=\left\{x\in\Omega \ : \ S^{+}(x)=S^{-}(x)  \right\}
		\end{equation*}
		as \textit{uniqueness set}. What we have proved in Theorem \ref{curveofcoincidence}, is that for every point there exists a lipschitz curve all contained in $\Omega$ except for its extreme points that is a geodesic for $d$ in the sense of Proposition \ref{existencecurvebondary}.
	\end{dfn}

	\section{Regularity of solutions on the uniqueness set}
	\label{section4}
	
	In this section we will study the regularity of solutions of \eqref{problem} in the uniqueness set. More precisely, we will show that such solutions of \eqref{problem} are both locally semiconcave and locally semiconvex at each point of the uniqueness set and therefore differentiable. Moreover, if the interior part of the uniqueness set is not empty, every solution is locally $C^{1,1}$. We still assume $\pO$ to be lipschitz.
	Let us recall some definitions (see \cite{CanSin2004}) .
	\begin{dfn}[Semiconcavity and semiconvexity (see Def 1.1.1 \cite{CanSin2004})]
		We say that $u_1,u_2:\Omega\to\R^d$ are respectively semiconcave and semiconvex, if there exist $C_1,C_2\ge 0$ such that
		\begin{align}
			&\label{semiconcavity}u(x+h)+u(x-h)-2u(x)\le C_1 |h|^2,\\
			&\label{semiconvexity}u(x+h)+u(x-h)-2u(x)\ge C_2 |h|^2,
		\end{align}
		for every $x\in \Omega$, $h\in\Rd$ such that the segment $[x-h,x+h]\subset \Omega$.\\
		We say that $u_1,u_2$ are respectively locally semiconcave and locally semiconvex, if for every $V\subset\subset\Omega$, there exist $C_1(V),C_2(V)\ge 0$ such that \eqref{semiconcavity} and \eqref{semiconvexity} hold for every $x\in V$, $h\in\Rd$ such that the segment $[x-h,x+h]\subset V$.
	\end{dfn}
	
	\begin{prp}\label{semiconcavity1}
		Let us assume that $\phizero\in C^2(\Omega\times\Rd\setminus\{0\})$. Then
		$S^{+}$ and $S^{-}$ are respectively locally semiconcave and locally semiconvex.
	\end{prp}
	\begin{proof}
		The proof of this fact is strongly inspired to the proof of Lemma 5.1 in \cite{Aro2009}. We prove the \eqref{semiconcavity} for $S^+$. The proof of the \eqref{semiconvexity} for $S^-$ is analogous.
		Let us consider $\delta>0$ such that $dist(x,\pO)>8\delta$ for all $x\in V$. Since $S^+$ is bounded, if $|h|\ge\delta$ the inequality \eqref{semiconcavity} is easily verified. Let us then assume that $|h|<\delta$. 
		Let $x$ in $V$ and $y_1\in\pO$, $\gamma_1\in\pathbar(y_1,x)$ such that 
		\[S^+(x)=g(y_1)+\int_{0}^{l(\gamma_1)}\varphi^0(\gamma_1(s),\dot{\gamma}_1(s))ds. \]
		We define $a:=\sup\{t\in[0,l(\gamma_1)] \ : \ d(\gamma_1 (t),\pO ) \le 4\delta \}$ and $y^{*}=\gamma(a)$. 
		We then consider
		\begin{equation*}
			y(s):=\frac{s-a}{\ell(\gamma_1)-a}\frac{h}{|h|} \quad \text{for } a\le s\le l(\gamma_1)
		\end{equation*}
		and
		\begin{equation*}\bar{\gamma}(s):=
			\begin{cases}
				\gamma_1(s) &\quad \text{for } 0\le s\le a\\
				\gamma_1(s)+ty(s)  &\quad \text{for } a\le s\le l(\gamma_1),
			\end{cases}
		\end{equation*}
		where $t\in[-|h|,|h|]$ is some parameter at our disposal. Notice that by construction $l(\gamma_1)-a>4\delta$ and that $\bar{\gamma}(s)\in\Omega$ for every $0\le s\le l(\gamma_1)$, indeed $d(\bar{\gamma}(s),\pO)>3\delta$  for $t\in[-|h|,|h|]$.
		
		We define
		\begin{align*}
			F(t):&=\int_{0}^{l(\gamma_{1})}\varphi^0(\bar{\gamma}(s),\dot{\bar{\gamma}}(s))ds\\&=\int_{0}^{a}\varphi^0(\gamma_1(s),\dot{\gamma_1}(s))ds+\int_{a}^{l(\gamma_1)}\varphi^0(\gamma_1(s)+ty(s),\dot{\gamma_1}(s)+t\dot{y}(s))ds
		\end{align*}
		and we observe that $F(0)=S^+(x)$ and $F^+(\pm|h|)\ge S^+(x\pm h)$. In order to prove the \eqref{semiconcavity} is then enough to prove that there exists $C_1\le0$ such that  
		\begin{equation*}
			F(|h|)+F(-|h|)-2F(0)\le |h|^2C_1.
		\end{equation*}
		Thanks to the assumptions on $\phizero$ we have
		\begin{equation*}
			F'(t)=\int_{a}^{l(\gamma_1)}\nabla_x\phizero(\bar{\gamma}(s),\dot{\bar{\gamma}}(s))\cdot y(s)+\nabla_p\phizero(\bar{\gamma}(s),\dot{\bar{\gamma}}(s))\cdot\dot{y}(s)ds
		\end{equation*}
		and 
		\begin{equation}
			F''(t)=\int_{a}^{l(\gamma_1)}\left<\nabla^2\phizero(\bar{\gamma}(s),\dot{\bar{\gamma}}(s))\cdot \left(y(s),\dot{y}(s)\right)^\intercal,\left(y(s),\dot{y}(s)\right)\right>,
		\end{equation}
		where $|F''(t)|\le \frac{C_1}{2}$ for some $C_1\ge 0$ for every $t\in[-|h|,|h|]$, by continuity of $D^2\phizero$.
		The proof is concluded thanks to the Taylor's formula with the rest of Lagrange, for which there exist $\theta_1\in [0,|h|]$ and $\theta_2\in [-|h|,0]$ such that
		\begin{equation*}
			F(|h|)+F(-|h|)-2F(0)=\left(F''(\theta_1)+F''(\theta_2)\right)|h|^2.
		\end{equation*}
	\end{proof}
	\begin{prp} \label{semiconcavequivalence}
		Let $u:V\to \R^d$ be a continuous function. Then the following facts are equivalent:
		\begin{enumerate}[(a)]
			\item \label{a} $u:V\to\R$ is semiconcave (semiconvex);
			\item \label{b} there exists $C\ge 0$ such that $u(x)-\frac{C}{2}|x|^2$ is (convex) concave in $V$;
			\item \label{c} u satisfies 
			\begin{equation}\label{semicontinuitylambda}
				(1-\lambda) u(x)+\lambda u(y)- u((1-\lambda)x+\lambda y)\le \lambda(1-\lambda)C|y-x|^2,
			\end{equation}
			for every $x,y\in V$ such that $[x,y]\subset V$ and $\lambda\in[0,1]$.
		\end{enumerate}  
	\end{prp}
	More details about the above result can be found in \cite{CanSin2004}.
	\begin{prp}\label{semiconcavity2}
		If $d(y,\cdot,)\in C^{1,1}_{\mathrm{loc}}(\bO)$, uniformly w.r.t. $y$, then
		$S^{+}$ is locally semiconcave. If  $d(\cdot,y)\in C^{1,1}_{\mathrm{loc}}(\bO)$ then $S^{-}$ is locally semiconvex.
	\end{prp}
	\begin{proof}
		We start proving that $d(y,\cdot)$ is locally semiconcave. Let $V\subset\subset \Omega$. Then by Proposition \ref{semiconcavequivalence} $d(y,\cdot)$ is semiconcave in $V$ uniformly in $y$ if there exists $C\ge0$ such that $d_{C}(y,x):=d(y,x)-\frac{C}{2}|x|^2$ is concave for every $x\in V$ and $y\in\bO$.
		By assumption there exists $C>0$ such that $|\nabla_xd(y,x)-\nabla_xd(y,z)|\le 2C|x-z|$. This implies that 
		\begin{align*}
			0&\ge \left\langle \nabla_xd(y,x)-\nabla_xd(y,z),x-z \right\rangle-2C|x-z|^2\\&=\left\langle \nabla_x\left(d(y,x)-C|x|^2\right)-\nabla_x\left(d(y,z)-C|z|^2 \right),x-z\right\rangle\\&=\left\langle \nabla_xd_{C}(y,x)-\nabla_xd_{C}(y,z),x-z\right\rangle.
		\end{align*}
		The monotonicity relation and the differentiability of $d_{C}(y,\cdot)$ imply that $d_{C}(y,\cdot)$ is concave for every $x\in V$ and $y\in\bO$.
		The semiconcavity of $S^+$ is proved just recalling that by Proposition \ref{maxminextensioncurves} there exists $y_1\in\pO$ such that $S^+(x)=g(y_1)+d(y_1,x)$.  The local semiconvexity of $S^-$ can be proved analogously.
	\end{proof}
	
	\begin{dfn}[see \cite{CanSin2004} Definition 3.1.1]\label{supersubdifferential}
		For any $x\in\Omega$, the sets
		\begin{align}
			&D^-u(x):=\left\{p\in\Rd \ : \ \liminf_{y\to x} \frac{u(y)-u(x)-\langle p,y-x\rangle}{|y-x|}\ge 0  \right\}\\
			&D^+u(x):=\left\{p\in\Rd \ : \ \limsup_{y\to x} \frac{u(y)-u(x)-\langle p,y-x\rangle}{|y-x|}\le 0  \right\}
		\end{align}
		are called respectively superdifferential and subdifferential of $u$ at $x$.
	\end{dfn}
	\begin{prp}\label{propappendix}[see \cite{CanSin2004} Proposition 3.3.1]
		If $u:\Omega\to\Rd$ is locally semiconcave, then $p$ belongs to $D^+u(x)$ if and only if there exists $C=C(V)$
		\begin{equation}
			\label{equivdefsemiconc}
			u(y)-u(x)-\langle p, y-x\rangle\le C|x-y|^2,
		\end{equation}
		for every $V\subset\subset \Omega$, $x,y\in V$ such that $[x,y]\subset V$.
	\end{prp}
	\begin{proof}
		If \eqref{equivdefsemiconc} holds, then clearly $p$ belongs to $D^+u(x)$. \\Viceversa, let us assume that $p\in D^+u(x)$. 
		The inequality \eqref{semicontinuitylambda} implies that 
		\begin{align*}
			\frac{u(y)-u(x)}{|y-x|}&\le \frac{u((1-\lambda)y+\lambda x)-u(x)}{(1-\lambda)|y-x|}+C\lambda|y-x|\\&\frac{u(x+(1-\lambda)(x-y))-u(x)}{(1-\lambda)|y-x|}+C\lambda|y-x|
		\end{align*}
		The thesis follows passing to the $\limsup$ for $\lambda\to 1^-$.
	\end{proof}
	\begin{thm}\label{regularity}
		Every solution $u$ of \eqref{problem} is differentiable at each point $x$ of $\mathcal{U}$, where $\mathcal{U}$ is the uniqueness set. Moreover $\nabla S^+=\nabla u= \nabla S^-$ on $\mathcal{U}$ and if $(x_k)\subset \mathcal{U}$ is a sequence of points that convergences to $x\in \mathcal{U}$, then $\nabla u(x_{k})\to \nabla u(x)$. Finally, if $\interior{\mathcal{U}}\neq\emptyset$ then $u\in C^{1,1}(\omega)$, for every $\omega$ subdomain of $\mathcal{U}$.
	\end{thm}
	\begin{proof}
		We provide here a sketch of the proof of the first part. For the last part we refer to Corollary 3.3.8 of \cite{CanSin2004}.\\
		\underline{Claim 1:} $D^+S^+(x)\neq\emptyset$ and $D^-S^-(x)\neq\emptyset$ for every $x\in\Omega$..\\
		\textit{Proof 1}. First of all we recall that if a function is locally semiconcave  in $\Omega$, then it is locally lipschitz (see for example Proposition 2.1.7 in \cite{CanSin2004}). Thus by Rademacher's theorem follows that $S^+$ is differentiable for a.e. x in $\Omega$. This means that for any $x\in \Omega$ there exists a sequence of points of differentiability $(x_k)$ that converges to $x$. By definition of subdifferential we have that $\nabla S^+(x_k)\in D^+S^+(x_k)$ for every $k$. Moreover, thanks to Proposition \ref{propappendix}, we have that 
		\begin{equation*}
			S^+(y)-S^+(x_k)-\langle \nabla S^+(x_k), y-x_k\rangle\le C|x_k-y|^2,
		\end{equation*}
		for every $y \in \Omega$ such that $[x_k,y]\in \Omega$. Then we pass to the limit and we use again Proposition \ref{propappendix}. In order to that $D^-S^-(x)\neq\emptyset$ it is enough to observe that $-S^{-}$ is locally semiconcave in $\Omega$ and that for any function $u$, $D^+(-u)=-D^-(u)$. \\
		\underline{Claim 2:} If $u$ is a solution of \eqref{problem}, then $u$ is differentiable at $x$ for every $x\in \mathcal{U}$ and $\nabla S^+(x)=\nabla S^-u(x)=\nabla u(x)$.\\
		\textit{Proof 2.} If $u$ is a solution of \eqref{problem}, then by Theorem \ref{curveofcoincidence} $u(x)=S^+(x)=S^-(x)$ for every $x\in \mathcal{U}$. Thus for every $V\subset\subset \Omega$ and $y\in\Omega $ with $[x,y]\subset V$ there exists $C=C(V)$ such that
		\begin{equation*}
			u(y)-u(x)\le S^{+}(y)-S^+(x)\le \langle\nabla S^+(x),y-x\rangle+ C|x-y|^2,
		\end{equation*}
		which by Proposition \ref{propappendix} implies that $\nabla S^+(x)\in D^+u(x)$.
		Analogously one can prove that $\nabla S^-(x)\in D^-u(x)\neq\emptyset$. Let $\theta$ be any unitary vector and $p^+\in D^+u(x)$, $p^-\in D^- u(x)$, then
		\begin{equation}
			\label{differentiability}
			\langle p^-,\theta  \rangle\le \liminf_{h\to0+}\frac{u(x+h\theta)-u(x)}{h}\le\limsup_{h\to0+}\frac{u(x+h\theta)-u(x)}{h}\le \langle p^+,\theta  \rangle,
		\end{equation}
		that implies $\langle p^--p^+, \theta  \rangle\le 0$. By the arbitrariness of the unitary vector $\theta$ we have that $p^+=p^-=\nabla u(x)=\nabla S^+(x)=\nabla S^-(x)$. \\
		\underline{Claim 3:} If $u$ is a solution of \eqref{problem}, then $\nabla u_{|\mathcal{U}}$ is continuous.\\
		\textit{Proof 3}. Let $x\in \mathcal{U}$. We consider a sequence $(x_k)\subset \mathcal{U}$ such that $x_k\to x$, then since $\nabla u (x_k)$ belongs both to $D^+u(x_k)$ and $D^-u(x_k)$ we have by Proposition \ref{propappendix} that
		\begin{align*}
			&u(y)-u(x_k)-\langle \nabla u(x_k), y-x_k\rangle\le C_1|x_k-y|^2,\\
			&u(y)-u(x_k)-\langle \nabla u(x_k), y-x_k\rangle\ge C_2|x_k-y|^2,
		\end{align*}
		for every $y\in V$ such that $[x_k,y]\in V$ and thus that $(\nabla u(x_k))$ is bounded. We prove the existence of the limit $\lim_{k\to \infty}\nabla u(x_k)$ arguing as in the \eqref{differentiability} for any cluster point.
	\end{proof}

	\section{An application to a special class of supremal variational problems}
	\label{section 5}
	In this section we will show that some classes of supremal variational problems, more precisely the one discussed in \cite{BriDep2022}, can be interpreted as special cases of our problem of constraints on the gradients. This allows for some regularity results of the absolute minimizers of those problems.
	Let $g\in \Winf(\Omega)\cap C(\bO)$ and consider the problem 
	\begin{equation}
		\label{supremal}
		\min\left\{F(v,\Omega):=\esup_{x\in\Omega}H(x,\nabla v(x)):v\in g+W^{1,\infty}(\Omega)\cap C_{0}(\Omega)\right\},
		\tag{S}
	\end{equation}
	where $H:\Omega\times\Rd\to\R$ is a Borel function satisfying the following ``natural" assumptions:
	\begin{enumerate}[(1)]
		\item $H\ge0$, $H(\cdot,0)=0$ and $H(x,\cdot)$ is quasi-convex, i.e. any sublevel $\{H(x,\cdot)\le\lambda\}$ is convex;
		\item the map $(x,p)\mapsto H(x,p)$ is uniformly (with respect to $x$) coercive in $p$, which means 
		\begin{equation*}
			\text{for every } \lambda\ge 0, \ \text{there exists } M\ge 0 \ \mbox{such that} \ \ \ H(x,p)\le\lambda \implies |p|\le M;
		\end{equation*}
		\item $H(x,\cdot)$ is lower semicontinuous for every $x\in\Omega$.
	\end{enumerate}
	The above assumptions are natural in order to provide a sufficient condition for the existence of minimizers of \eqref{supremal}. In particular, lower semicontinuity and quasi-convexity imply sequential lower semicontinuity of the supremal functional $F(u):=\esup_{x\in\Omega}H(x,\nabla u(x))$ with respect to the weak* topology of $\Winf(\Omega)$ (see Th. 3.4 in \cite{BarJenWan2001}). Furthermore, we observe that  finding a minimizer for \eqref{supremal} is equivalent to finding a solution of 
	\begin{equation*}
		\begin{cases}
			u=g \quad &\mbox{on} \ \pO,\\
			\gradu(x)\in K(x)=\left\{p \ :\ H(x,p)\le\mu\right\} \quad &\mbox{for a.e. }x \in \Omega.
		\end{cases},
	\end{equation*}
	where 
	\begin{equation*}
		\mu:=\min\left\{F(v,\Omega):=\esup_{x\in\Omega}H(x,\nabla v(x)):v\in g+W^{1,\infty}(\Omega)\cap C_{0}(\Omega)\right\},
	\end{equation*}
	indeed by (1),(2) and (3) we have that $ K(x)=\left\{p \ :\ H(x,p)\le\mu\right\}$ is convex and compact.  
	In addition to (1),(2),(3), we also require that $H$ satisfies the following properties (the same required in \cite{BriDep2022}):  
	\begin{enumerate}[(4)]
		\item For all $\lambda_{1}>\lambda_{2}\ge 0$ 
		there exists $\alpha>0$ such that 
		\begin{equation*}
			\left\{H(x,\cdot)\le\lambda_{2}\right\}+B(0,\alpha)\subset\{H(x,\cdot)\le\lambda_{1}\}, \quad \mbox{for all }x\in\Omega;
		\end{equation*}
	\end{enumerate}
	\begin{enumerate}[(5)]
		\item $(x,p)\mapsto H(\cdot, p)$ is continuous for every $(x,p)\in\Omega\times\Rd$.
	\end{enumerate}
	\begin{enumerate}[(6)]
		\item the interior part of the level set $\{H(x,p)=\lambda\}$ is empty for every $\lambda\ge 0$.
	\end{enumerate}
	
	In this setting we have that $$K(x):=\left\{p \ :\ H(x,p)\le\mu\right\} \  \mbox{and } \varphi(x,p):=\inf\left\{t>0 \ \big| \ \frac{p}{t}\in \left\{p \ :\ H(x,p)\le\mu\right\}\right\}$$ satisfy the assumptions given on $K$ and $\varphi$ (\eqref{alphaM} and \eqref{continuityphi}) at the beginning of Section \ref{section2}. \\
	In particular assumptions $(2)$ and $(4)$ imply property \eqref{alphaM} of $K(x)$ and $(5)$ implies continuity (property \eqref{continuityphi}) of $x\mapsto\varphi(x,p)$. In order to prove this last fact thanks to Proposition \ref{hausdorffdistance}, one can equivalently show that
	\[\left\{p \ :\ H(x_n,p)\le\mu\right\}\overset{H}{\longrightarrow}\left\{p \ :\ H(x,p)\le\mu\right\}, \quad \text{for every }x_n\to x \]
	with the following sketched steps:
	\begin{itemize}
		\item for every $\eps>0$, by continuity of $H$ we have that
		\begin{align*}
			&\left\{p \ :\ H(x_n,p)\le\mu\right\}\subset\left\{p \ :\ H(x,p)\le\mu+\eps\right\} \ \text{and}\\  &\left\{p \ :\ H(x,p)\le\mu\right\}\subset\left\{p \ :\ H(x_n,p)\le\mu+\eps\right\}
		\end{align*}
		for $n$ big enough;
		\item $\left\{p \ :\ H(x,p)\le\mu+\eps\right\}\subset\left\{p \ :\ H(x,p)\le\mu\right\}+B(0,f(\eps))$, for some $f$, $f(\eps)\to 0$ when $\eps\to 0$,  for example considering $$f(\eps):=\sup_{q\in \partial\left\{p \ :\ H(x,p)\le\mu+\eps\right\}}d(q,\left\{p \ :\ H(x,p)\le\mu\right\}).$$
	\end{itemize}
	Assumption (6) insure that $H(x,p)=\mu$ implies $\varphi(x,p)=1$ and it will be used for the proof of Proposition \ref{Hregularphiregular} and \ref{attainincludeduniq}.\\
	Finally we can define in this setting also:
	\begin{align*}
		&\phizero(x,q):=\sup\left\{p\cdot q  \ : \  H(x,p)\le\mu \right\},\\ &d(x,y):=\inf\left\{\int_{0}^{1}\phizero(\xi(t),\dot{\xi}(t))dt: \xi\in \mathsf{path}(x,y) \right\}.
	\end{align*}
	From now on we will consider $H:\Omega\times\Rd\to\Rd$ satisfying the assumptions $(1)$-$(5)$.
	\begin{prp}
		Let $H:\Omega\times\Rd\to\Rd$ as above. Given a boundary datum $g\in\WinfCo$, if there exists a point $x\in\Omega$ where $S^{+}(x)=S^{-}(x)$ then there is a lipschitz curve $\gamma$ passing trough $x$ such that $\gamma$ belongs to the uniqueness set.
	\end{prp}
	\begin{proof}
		The proof follows directly by Theorem \ref{curveofcoincidence}.
	\end{proof}
	The following result is not directly useful for the purpose of this paper, but it may be interesting for a better understanding of the functional $H$ and its regularity.
	\begin{prp}
		\label{Hregularphiregular}
		If $H(x,p)$ is positively $1$-homogeneous w.r.t. $p$ and $\nabla_{x} H(x,p)\in C(\Omega\times\Rd)$, then $\varphi(x,p)\in C(\Omega\times\Rd)$.
	\end{prp}
	\begin{proof}
		We first consider $p$ such that $\varphi(x,p)=1$ (that means $H(x,p)=\mu$). We have to prove that 
		\begin{equation}
			\label{phiderivable}
			\lim_{|h|\to 0}\frac{\varphi(x+h,p)-\varphi(x,p)-\grad_{x}\varphi(x,p)\cdot h}{|h|}=0
		\end{equation}
		and $\grad_{x}\varphi(x,p)$ is continuous.
		Since $\grad_{x}H$ exists and is continuous, it holds
		\begin{equation*}
			H(x+h,p)=H(x,p)+\grad_{x}H(x,p)\cdot h+o(|h|).
		\end{equation*}
		By $1$-homogeneity, we have 
		\begin{equation*}
			\frac{\mu}{\mu+\grad_{x}H(x,p)\cdot h+o(|h|)}H(x+h,p)=H\left(x+h,p\frac{\mu}{\mu+\grad_{x}H(x,p)\cdot h+o(|h|)} \right)=\mu.
		\end{equation*}
		That implies, thanks to assumption (6),
		\begin{equation*}
			1=\varphi(x+h,p\frac{\mu}{\mu+\grad_{x}H(x,p)\cdot h+o(|h|)})=\varphi(x+h,p)\frac{\mu}{\mu+\grad_{x}H(x,p)\cdot h+o(|h|)}.
		\end{equation*}
		From which we get
		\begin{equation*}
			\varphi(x+h,p)=1+\frac{\grad_{x}H(x,p)\cdot h+o(|h|)}{\mu}=\varphi(x,p)+\frac{\grad_{x}H(x,p)\cdot h+o(|h|)}{\mu}.
		\end{equation*}
		The \eqref{phiderivable} and the regularity of $\nabla_{x}\varphi $ are proved just observing that $\nabla_{x}\varphi(x,p)=\mu^{-1}\grad_{x}H(x,p)$.\\
		If $\varphi(x,p)\not=1$, we can use the same reasoning with $\frac{p}{\varphi(x,p)}$ and, exploiting the $1$-homogeneity of $\varphi$ w.r.t. to $p$, we get that  $\nabla_{x}\varphi(x,p)=\mu^{-1}\varphi(x,p)\grad_{x}H(x,p)$.
	\end{proof}
	
	\begin{crl}
		If $H(x,p)$ is positively $1$-homogeneous w.r.t. $p$, $\nabla_{x} H:\Omega\times\Rd\to\R$ is continuous and the sublevel sets $\{H(x,p)\le\mu\}$ are strictly convex for every $x\in\Omega$, then $\phizero(\cdot,q)\in C^1(\Omega)$ for any $q\in\Rd$.
	\end{crl}
	\begin{proof}
		The proof follows directly by Proposition \ref{Hregularphiregular} and \ref{differentiabilityphizero}.
	\end{proof}
	The aim now is to show for this case further properties of the uniqueness set $U$ defined in Section \ref{section3}.
	
	\begin{prp}
		\label{uniqincludedattain}
		Let $u\in\WinfCo$ be an optimal solution of \eqref{supremal}. If $x$ belongs to the uniqueness set $\mathcal{U}$ and $u$ is differentiable at $x$ then 
		\begin{equation*}
			H(x,\nabla u (x))=\esup_{x\in\Omega}H(x,\gradu(x)).
		\end{equation*}
	\end{prp}
	\begin{proof}
		By Theorem \ref{regularity} $u$ is differentiable at $x$. Since $u$ is a solution of \eqref{supremal}, we know that $\gradu(x)\in\left\{H(x,p)\le \mu \right\}$. Let $\gamma$ be a curve of the uniqueness set passing trough $x$. Let $s>0$ such that $x=\gamma(x)$. By Theorem \ref{curveofcoincidence} we know that 
		\begin{equation*}
			\gradu(x)\cdot\dot{\gamma}(s)=\lim_{t\to 0}\frac{u(x+t\dot{\gamma}(s))-u(x)}{t}=\phizero(x_{0},\dot{\gamma}(s)),
		\end{equation*}
		as stated in the last part of Theorem \ref{curveofcoincidence}.
		Then by Lemma \ref{phizeropuntuale} we have that $\gradu(x
		)$ belongs to the boundary of $K(x)=\left\{H(x,p)\le \mu \right\}$, that means $H(x,\gradu(x))= \mu $.
	\end{proof}
	\begin{rmk}
		Under the same assumption of Proposition \ref{semiconcavity1} on $\phizero$ or of Proposition \ref{semiconcavity2} on $d$, by Theorem \ref{regularity} if $x$ belongs to the uniqueness set then a solution $u$ of \eqref{supremal} is differentiable at $x$.
	\end{rmk}
	Without further regularity assumptions, it is also possible to define the functional $x\mapsto H(x,\nabla u)(x)$ for every $x\in\Omega$, in such a way that this new definition extends the definition of $x\mapsto H(x,\nabla u(x))$ to the points where $u$ is not differentiable (see Definition 4.1 in \cite{BriDep2022}). 
	\begin{dfn}\label{punctualextension}
		Let $u\in \Winf(\Omega)\cap C(\bO)$. For any $x_{0}\in\Omega$ and for any $r>0$ such that $r<dist(x_{0},\partial \Omega)$, we set
		\begin{equation}
			\mu(x_{0},r):=\inf\{\nu: u(x)-u(x_{0})\le d_{\nu} (x_0,x) \ \mbox{for any} \  x\in B(x_{0},r)\},
		\end{equation}
		Where $d_{\nu} (x_0,x)$ is the distance defined as in Definition \ref{distance} associated to the convex set $K_{\nu}(x):=\{p \ : \ H(x,p)\le \nu\}$.\\
		We observe that $\mu(x_{0},r)$ is not decreasing in $r$. This allows for the following definition:
		\begin{equation*}
			H(x_{0},\nabla u)(x_{0}):=\lim_{r\rightarrow 0}\mu(x_{0},r)= \inf_r \mu (x_0,r).
		\end{equation*}
	\end{dfn}
	\begin{rmk}
		If $u$ is differentiable in $x_0$, then $H(x_0,\nabla u (x_0))=H(x_0,\nabla u)(x_0)$ (see Proposition 4.4 and 4.5 in \cite{BriDep2022}).
	\end{rmk}
	With this definition we are able to prove the following improvement of the Proposition \ref{uniqincludedattain}:
	\begin{prp}
		
		\label{uniqincludedattainbis}
		Let $u\in\WinfCo$ be an optimal solution of the \eqref{supremal}. If $x$ belongs to the uniqueness set $\mathcal{U}$ then 
		\begin{equation*}
			H(x,\nabla u) (x)=\esup_{x\in\Omega}H(x,\gradu(x))=\mu.
		\end{equation*}
	\end{prp}
	\begin{proof}
		We have to prove that $\lim_{r\to 0}\mu(x,r)=\mu$. For every $r>0$, $\mu(x,r)\le \mu$, by Proposition \ref{1lipu}. Moreover, if we consider a curve $\gamma\subset \mathcal{U}$ passing through $x$ given by Theorem \ref{curveofcoincidence}, we have that $u(\gamma(s))-u(x)=d(x,\gamma(s))$, proving that $\mu(x,r)=\mu$ for every $r$. 
	\end{proof}
	Minimizing a supremal functional is a global problem rather than a local one. This leads often to lack of uniqueness if the minimizers, which in general are not minimizers if one restrict the problem to an open subset of the initial set. That is why it is natural to consider the so-called class of absolute minimizers instead of the bigger class of optimal solutions of \eqref{supremal}. For reader's convenience we provide below the definition.
	\begin{dfn}\label{AM}
		An absolute minimizer for \eqref{supremal} is a function $u\in W^{1,\infty}(\Omega)\cap C(\bO)$ such that $u=g$ on $\pO$ and for all open subset $V\subset\subset \Omega$ one has 
		\begin{equation*}
			\esup_{x\in V} H(x,\nabla u(x))\le\esup_{x\in V}H(x,\nabla v(x))
		\end{equation*}
		for all $v$ in $W^{1,\infty}(V)\cap C(\bV)$ such that $u=v$ on $\pV$.\\
		We remark that an absolute minimizer is an optimal solution of \eqref{supremal} (see \cite{ChaDep2007}, Lemma B.1).
	\end{dfn}

	In \cite{BriDep2022} the authors also define the \textit{attainment set} of a function $u$, $\mathcal{A}(u)$, as the set  
	\begin{equation*}
		\mathcal{A}(u):=\left\{x\in\Omega \ : \ H(x,\nabla u)(x)=\esup_{x\in\Omega}H(x,\gradu(x))=\mu  \right\}.
	\end{equation*}
	The interest of the attainment set $\mathcal{A}(u)$ lies in the fact, proved in \cite{BriDep2022}, that if $u$ is an absolute minimizer for \eqref{supremal}
	\[\mathcal{A}(u)\subset\mathcal{A}(v), \quad \text{for any }v \ \text{solution of \eqref{supremal}}.\]
	Therefore what we have in fact proved in Proposition \ref{uniqincludedattain} is that $\mathcal{A}(u)\subset \mathcal{U}$, for any $u$ solution of \eqref{supremal}. In the next theorem we show that if we consider an absolute minimizer $u$, the converse inclusion also holds, that means that $\mathcal{A}(u)= \mathcal{U}$.
	The result is based on Theorem 5.3 in \cite{BriDep2022}, which requires assumption (6) on $H$ (as shown in the Example 5.4 in \cite{BriDep2022}).
	\begin{thm}
		\label{attainincludeduniq}
		Let $u\in\WinfCo$ be an absolute minimizer \eqref{supremal} and let $x\in\Omega$ be such that $H(x,\nabla u) (x)=\esup_{x\in\Omega}H(x,\gradu(x))$. Then $x$ belongs to the uniqueness set. 
	\end{thm}
	\begin{proof}
		Let $x$ be as in the statement. By Theorem 5.3 in \cite{BriDep2022} we know that there exist $y_{1},y_{2}\in\pO$ such that $u(x)=g(y_{1})-d(x,y_{1})$ and $u(x)=g(y_{2})+d(y_{2},x)$. Then $S^{+}(x)\le u(x)\le S^{-}$, by definition of $S^{+}$. Moreover, by  Proposition \ref{solutions}, $S^{+}\ge u(x)\ge S^{-}$. 
	\end{proof}
	\begin{crl}
		Under the same assumption of Proposition \ref{semiconcavity1} on $\phizero$ or of Proposition \ref{semiconcavity2} on $d$, any absolute minimizer $u$ for the problem \eqref{supremal} is differentiable on $\mathcal{U}(u)$ and $\nabla _{|\mathcal{A}(u)}$ is continuous. Moreover if $\interior{\mathcal{A}(u)}\neq\emptyset$, then $u\in C^{1,1}_{\text{loc}}(\mathcal{A}(u))$.
	\end{crl}
	\begin{proof}
		The proof follows by Theorem \ref{attainincludeduniq} and Theorem \ref{regularity}.
	\end{proof}
	
	\appendix
	\section{}
	\label{appendix}
	\begin{prp}
		\label{dualitysuppfunc}
		Since $K(x)$ (and consequently $\varphi(x,\cdot)$) is convex for every $x\in\bO$, then $\varphi(x,p)=\varphi^{00}(x,p)$, for every $x\in\bO$ and $p\in\Rd$, where $\varphi^{00}$ is the support function of $\phizero$, i.e.
		\begin{equation*}
			\varphi^{00}(x,p):=\sup\{p\cdot q \ : \phizero(x,q)\le 1 \}=\sup_{q\not=0}\left\{\frac{p\cdot q}{\phizero(x,q)} \right\}.
		\end{equation*}
	\end{prp}
	\begin{proof}
		Let $x\in\bO$ and $p\in\Rd$. It's not difficult to show that $\varphi^{00}(x,p)\le \varphi(x,p)$, indeed
		\begin{equation*}
			\frac{p\cdot q}{\varphi(x,p)}\le \phizero(x,q),
		\end{equation*} 
		for every $q\in\Rd$. Then 
		\begin{equation*}
			\sup_{q\not=0}\left\{\frac{p\cdot q}{\phizero(x,q)} \right\}\le \varphi(x,p).
		\end{equation*}
		Let us prove the converse inequality: $\varphi^{00}(x,p)=\lambda\Rightarrow\varphi(x,p)\le \lambda$. Thanks to the positively $1$-homogeneity of $\varphi(x,\cdot)$ and $\varphi^{00}(x,\cdot)$, we can equivalently show that 
		\begin{equation*}
			\varphi^{00}(x,p)=1 \Rightarrow p\in K(x).
		\end{equation*}
		Let us assume by contradiction that there exists $p$ such that $p\cdot q\le 1$ for all $q$ such that $\phizero(x,q)\le 1$ and $p\notin K(x)$. By the Hahn-Banach Theorem, there exists a closed hyperplane $\{v \ : \ f\cdot v=\alpha\}$ such that 
		\begin{equation*}
			f\cdot v < \alpha < f \cdot p, \quad \mbox{for all } v\in K(x).
		\end{equation*}
		If we take $\bar{f}=\phizero(x,f)^{-1}f$, then $\phizero(x,\bar{f})=1$ and we get 
		\begin{equation*}
			\bar{f}\cdot v < \frac{\alpha}{\phizero(x,f)} < \bar{f} \cdot p\le 1, \quad \mbox{for all } v\in K(x),
		\end{equation*}
		so that $\bar{f}\cdot v<1$ for all $v\in K(x)$. On the other hand, by Proposition \ref{phizeropuntuale} we know that there exists $\bar{v}\in K(x)$ such that $\bar{f}\cdot\bar{v}=\phizero(x,\bar{f})=1$ and we get a contradiction. 
	\end{proof}
	\begin{lmm}
		\label{approximationbytrasversalcurves}
		Let $x,y\in\Omega$, $\gamma\in\Winf((0,1))\cap C([0,1])$ and $E$ such that $\Leb^{n}(E)=0$. Then for every $\eps>0$ there exists a curve $\gamma_{\eps}$ transversal to $E$ (i.e. $\mathcal{H}^{1}(\gamma_{\eps}((0,1))\cap E)=0$) such that 
		\begin{equation*}
			||\gamma_{\eps}-\gamma||_{\Winf((0,1))}<\eps.
		\end{equation*}
	\end{lmm}
	\begin{proof}
		Let $g(t)\in C^{1}[0,1]$ be a non negative function such that $g(0)=g(1)=0$. For every $v\in\mathbb{R}^{n}$, we define the curve $\gamma_{v}(t)=\gamma(t)+vg(t)$. Let $A$ be the set of the points $(t,v)\in [0,1]\times\mathbb{R}^{n}$ such that $\gamma_v(t)\in E$ and $A_{t}:=\{v\in\mathbb{R}^{n}: (t,v)\in A\}$. Since $\Leb^{n}(E)=0$, $\Leb^{n}(A_{t})=0$ for every fixed $t\in[0,1]$ and therefore also $\Leb^{n+1}(A)=0$. This implies that $A_{v}:=\{t\in[0,1]: (t,v)\in A\}$ is such that $\Leb^{1}(A_{v})=0$ for a.e. $v\in\mathbb{R}^{n}$. Let $v$ such that $\Leb^{1}(A_{v})=0$, then $\gamma_v$ is transversal to $E$. Indeed by lipschitzianity of $\gamma_{v}$ we have that $\mathcal{H}^{1}(\gamma_{v}(A_{v}))=0$ and we conclude by observing that $\gamma_{v}(A_{v})=\gamma_{v}([0,1])\cap E$.\\
		Finally, we choose $v$ such that $|v|<{\eps}/{||g(t)||_{\Winf([0,1])}}$, so to obtain
		\begin{equation*}
			||\gamma_{v}-\gamma||_{\Winf((0,1))}\le|v||g||_{\Winf((0,1))}<\eps.
		\end{equation*}
	\end{proof}
	\begin{lmm}
		\label{approxlipsets}
		Let $B_{1}^{d-1}(0)\subset\R^{d-1}$, the open ball of radius $1$ and center $0$ in $\R^{d-1}$. Let $f:B_{1}^{d-1}(0)\to\R$ a lipschitz function. Let $\Omega=\left\{(x',x_{d}) \ : \ x'\in B_{1}^{d-1}(0), x_{d}>f(x')  \right\}$. Then for every lipschitz curve $\gamma:[0,1]\to\bO$ there exists a sequence of curves $(\gamma_{n})$, such that 
		\begin{enumerate}[1.]
			\item $\gamma_{n}([0,1])\subset\Omega$;
			\item $\gamma_{n}\to\gamma$ uniformly;
			\item  $\dot{\gamma}_{n}\to\gamma$ a.e..
		\end{enumerate}
	\end{lmm}
	\begin{proof}
		We take a sequence of functions $(f_{n}):B_{1}^{d-1}(0)\to\R$ such that 
		\begin{itemize}
			\item $f_{n}\in C^{1}(B_{1}^{d-1}(0))$ for every $n\in\mathbb{N}$;
			\item $f_{n}\to f$ in $\Winf((B_{1}^{d-1}(0))$;
			\item $f_{n}>f$.
		\end{itemize}
		Let $\pi((x',x_{d})=x'$ the projection of $\Rd$ in $\R^{d-1}$. We define 
		\begin{equation*}
			\gamma_{n}(s):=\gamma(s)+\left(0, f_{n}(\pi(\gamma(s)))  \right)-\left(0, f(\pi(\gamma(s))). \right)
		\end{equation*}
		Clearly $\gamma_{n}([0,1])\subset\Omega$ and $\gamma_{n}\to\gamma$ uniformly. Moreover
		\begin{align*}
			&\dot{\gamma}_{n}(s)=\dot{\gamma}(s)+(0, \left<\grad f_{n}(\pi\circ\gamma),\grad\pi(\gamma(s))\dot{\gamma}(s)\right>) -(0, \left<\grad f(\pi\circ\gamma),\grad\pi(\gamma(s))\dot{\gamma}(s)\right>)=\\
			&=\dot{\gamma}(s)+ (0, \left<\grad f_{n}(\pi \circ \gamma)- \grad f(\pi\circ\gamma),\pi(\dot{\gamma}(s))\right>).
		\end{align*}
		Then 
		\begin{equation*}
			|\dot{\gamma}_{n}(s)-\gamma(s)|\le \left|\left|\grad f_{n}-\grad f\right|\right|_{L^{\infty}}|\dot{\gamma}(s)|.
		\end{equation*}
	\end{proof}
	
	\section*{Acknowledgement}
	The research of the author is partially financed  by the {\it ``Fondi di ricerca di ateneo, ex 60 $\%$''}  of the  University of Firenze and is part of the project  {\it "Alcuni problemi di trasporto ottimo ed applicazioni"}  of the  GNAMPA-INDAM.
	
	\bibliographystyle{plane}
	
\end{document}